\documentclass[12pt, letterpaper]{amsart}
\usepackage{fullpage} 
\usepackage{tikz} 
\usepackage{amsmath}
\usepackage{hyperref}
\usepackage{tcolorbox}
\usepackage{amsfonts, amsmath, amssymb, amsthm, hyperref, enumitem, fancyhdr, color, soul, comment, standalone, tcolorbox}
\usepackage{hyperref}
\usepackage[nocompress]{cite}
\hypersetup{
    colorlinks,
    citecolor=black,
    filecolor=black,
    linkcolor=black,
    urlcolor=black
}
\newtheorem{lemma}{Lemma}
\newtheorem{theorem}{Theorem}
\newtheorem{remark}{Remark}
\newtheorem{corollary}{Corollary}
\newtheorem{definition}{Definition}
\newtheorem{proposition}{Proposition}
\newtheorem{hypothesis}{Hypothesis}

\newcommand{\cal}[1]{\mathcal{#1}}

\allowdisplaybreaks
\title{\vspace{-1 in} On Lattice Points, Short-Time Estimates, and Global Well-posedness of the Quintic NLS on $\mathbb{T}$}
\author{Ryan McConnell}
\thanks{The first author was partially supported by the NSF grant DMS-2154031.}


\begin{document}
\begin{abstract}
    We prove and utilize an improvement to the short time estimates of Burq, G\'erard, \& Tzvetkov \cite{BurqShortTime} on $\mathbb{T}$ via connecting this estimate to the number of lattice points in thin annuli. As a consequence, we enhance the global well-posedness level of the periodic quintic Nonlinear Schr\"odinger equation to $s > \frac{131}{624}\sim 0.21$, which is an improvement on the results of De Silva, Pavlovi\'c, Staffilani, \& Tzirakis, \cite{de2007global}, Li, Wu, \& Xu, \cite{li2011global}, and Schippa \cite{schippa2024improved}. We also present conditional results, dependent on improvements on the count of lattice points in thin annuli.
    \end{abstract}
\maketitle
\tableofcontents

\section{Introduction}
Consider the periodic mass critical Nonlinear Schr\"odinger equation (NLS) posed on $\mathbb{T}$
\begin{equation}\label{Equation: Quintic NLS}
    \begin{cases}
        (i\partial_t+\Delta)u = \pm |u|^4u\\
        u(x,0) = u_0(x)\in H^s(\mathbb{T}),
    \end{cases}
\end{equation}
where the equation is \textit{defocusing} when $+$ is taken, and \textit{focusing} when $-$ is taken. Equation \eqref{Equation: Quintic NLS} is mass critical in the sense that the $L^2_x(\mathbb{T})$ norm is scale invariant-- if $u$ solves \eqref{Equation: Quintic NLS}, then $u^{\lambda}(x,t) := \lambda^{-\frac12}u(\lambda^{-1} x, \lambda^{-2} t)$ solves \eqref{Equation: Quintic NLS} posed on $\lambda \mathbb{T}$ and
\[
\|u\|_{L^2_x(\mathbb{T})} = \|u^{\lambda}\|_{L^2_x(\lambda\mathbb{T})}.
\]
While this doesn't respect the geometry of $\mathbb{T}$ like it does for $\mathbb{R}$, this still serves to guide intuition. 

Well-posedness of \eqref{Equation: Quintic NLS} is immediate for $s > \frac12$, while Bourgain \cite{Bour1}, with his proof of the $L^6_{x,t}$ Strichartz estimate, proved local well-posedness for $s > 0$. Even though the question of local well-posdness is settled, global existence for the full range $0 < s < 1$ has been a lingering question for $30$ years.

With an eye towards global well-posedness, we observe that \eqref{Equation: Quintic NLS} satisfies conservation of Mass and Energy
\begin{equation*}
\|u(t)\|_{L^2_x} = \|u_0\|_{L^2_{x}},\qquad E(u)(t) := \frac{1}{2}\int_{\mathbb{T}}|\partial_x u(t)|^2\,dx\pm \frac{1}{6}\int_{\mathbb{T}}|u(t)|^{6}\,dx = E(u_0),
\end{equation*}
which naturally leads to global existence for the defocusing (as well as the small data focusing) case when $s \geq 1$, yet global well-posedness between these two conservation laws, for $0 < s < 1$, is a more subtle question. Progress first began with the introduction of Bourgain's High-Low decomposition \cite{bourgain1998refinements}, which requires the presence of a smoothing property for high-frequency data. This method was refined by Colliander, Keel, Staffilani, Takaoka, and Tao, who introduced the $I-$method in \cite{colliander2001global, MR1906069, colliander2003sharp}, and which was subsequently applied with varying success to the equation \eqref{Equation: Quintic NLS}, in \cite{de2007global, li2011global}. Specifically, \cite{de2007global} proved global well-posedness for $s > \frac49$, while \cite{li2011global} improved this result to $s > \frac25$ by utilizing a more refined modified energy (with the obvious caveat of small data for the focusing variant).

Now, the $I$-method is the inclusion of an operator $I$ defined by a smooth, decreasing, and radial multiplier:
\begin{equation}\label{Definition: m definition}
m(|n|) = 
\begin{cases}
    1 & |n|\leq N\\
    N^{1-s}|n|^{s-1} & |n|> 2N,
\end{cases}
\end{equation}
whose effect is to place $u_0\in H^s(\mathbb{T}^d)$ into $H^1(\mathbb{T})$, where we can attempt to use the energy functional above. The standard $I$-method argument then obtains quantitative bounds (in terms of $N$) on the energy increment associated to the flow of $Iu$, which is used to obtain global well-posedness. 

The arguments of {De Silva}, Pavlovi{\'c}, Staffilani, \& Tzirakis \cite{de2007global} as well as Li, Wu, \& Xu \cite{li2011global} supplement the already available factors of $N$ with those following from a refined bilinear estimate of the following form: if $f,g\in L^2(\lambda\mathbb{T})$, $M\gg L$, and $\lambda\gg 1$, then:
\begin{equation}\label{Equation: Base Introduction Bilinear}
\big\|e^{it\Delta_{\lambda\mathbb{T}}}P_Mfe^{it\Delta_{\lambda\mathbb{T}}}P_Lg\big\|_{L^2_{x,t}(\lambda\mathbb{T}\times [-1,1])}\lesssim \left(\frac{1}{\lambda}+\frac{1}{1+M}\right)^{1/2}\|P_Mf\|_{L^2_x(\lambda\mathbb{T})}\|P_Lf\|_{L^2_x(\lambda\mathbb{T})},
\end{equation}
where $\Delta_{\lambda\mathbb{T}}$ is the Laplace-Beltrami operator associated to the $\lambda-$torus. It's worth noting that this above result should be compared to the bilinear estimate on $\mathbb{R}$, where we view $\lambda\mathbb{T}\to\mathbb{R}$ as $\lambda\to\infty$.

However, due to the resonant structure of the quintic NLS, it appears that one cannot hope to obtain any improvements via the introduction of higher order modified energies as in \cite{colliander2003sharp}. Furthermore, as the the bilinear estimate above is sharp, one appears stuck without the introduction of a different idea, which was originally observed by Hani \cite{HaniClosedGWPNLS} in connection to the global well-posedness of the cubic NLS on closed manifolds. The key idea was that the low frequency portion of the scaled short-time estimate should hold for much longer times than $T\sim 1$; specifically, the short-time estimates of Burq, G\'{e}rard, \& Tzvetkov \cite{BurqShortTime} could be used in a scaling critical norm to establish longer times of existence that depend on $\lambda$, circumventing some of the blow up when $s\to 0$ in $\lambda \sim N^\frac{1-s}{s}$. The addition of this idea, together with the fine properties of the structure of $\mathbb{T}^d$, was observed by Herr \& Kwak\footnote{They technically did not need to rely on the $I-$method for their proof, as they obtained a short time Strichartz estimate with logarithmically dependent time.} \cite{herr2023strichartz} and subsequently by Schippa \cite{schippa2024improved}, to significantly improve the global well-posedness threshold.

The fundamental building block is the family of estimates given in \cite{HaniClosedGWPNLS}, which hold on timescales $\sim N^{-1}$:
\begin{equation}\label{Equation: Generic short time estimate}
\|e^{it\Delta_\mathcal{M}}P_{\leq N} f\|_{L^p_tL^q_x(N^{-1}\mathbb{T}\times\mathcal{M})}\lesssim \|P_{\leq N} f\|_{L^2(\mathcal{M})},
\end{equation}
whenever $(p,q)$ is an admissible exponent pair for the Schr\"odinger equation on $\mathbb{R}^d$, with $d$ the dimension of the (Riemannian) compact manifold $\mathcal{M}$. The key insight of \cite{herr2023strichartz}, then, is that the short-time estimates above hold on generic compact manifolds, so we should possibly expect such estimates to hold on longer timescales when restricted to $\mathbb{T}^d$, culminating in the estimate:
\begin{equation}\label{Equation: Herr Short time estimate}
\|e^{it\Delta_{\mathbb{T}^2}}P_{\leq N} f\|_{L^4_{t,x}(N^{-\alpha}\mathbb{T}\times \mathbb{T}^2)}\lesssim_\alpha \|P_{\leq N} f\|_{L^2_x(\mathbb{T}^2)},\quad\,\forall\alpha > 0.
\end{equation}
This estimate was utilized by \cite{herr2023strichartz} to obtain small data global well-posedness of the cubic NLS on $\mathbb{T}^2$ for $s > 0$ (and subsequently an analagous result for data in $L^2(\mathbb{T}^2)$ \cite{herr2025global}), while \eqref{Equation: Generic short time estimate} was used by \cite{schippa2024improved} to obtain global well-posedness for the quintic NLS on $\mathbb{T}$ for $s > \frac13$.

However, the threshold $s > \frac{1}{3}$ of \cite{schippa2024improved} is imposed by the duration in \eqref{Equation: Generic short time estimate}, and is the key avenue along which to attack this problem. This manuscript improves this estimate, which provides the following improvement to the global well-posedness threshold.
\begin{theorem}\label{Theorem: GWP theorem}
    Let $\alpha > \frac{131}{208}$ and $s > \frac{\alpha}{3}$. Then \eqref{Equation: Quintic NLS} is globally well-posed in $H^s(\mathbb{T})$ for any $\|u_0\|_{L^2_{x}(\mathbb{T})}\ll 1$.
\end{theorem}

As previously mentioned, Theorem \ref{Theorem: GWP theorem} follows from an improvement to the short time estimate \eqref{Equation: Generic short time estimate} when posed on $\mathbb{T}$, which is derived from a connection between \eqref{Equation: Generic short time estimate} and lattice counting in thin annuli. However, due to scaling considerations discussed in Section \ref{Section: Basic Multilinear Estimates}, Theorem \ref{Theorem: GWP theorem} also requires the use of a trilinear estimate, Proposition \ref{Proposition: Trilinear Free Estimate}, over the bilinear estimate \eqref{Equation: Base Introduction Bilinear}. While this estimate has loss, the loss in many situations can be taken to be in the smallest frequency, and hence we can use this estimate in much the same way as we utilize the bilinear estimate. In fact, much of the construction of the modified energies and pointwise/regularity estimates, outside of a few tighter bounds, follows in exactly the same manner as in \cite{schippa2024improved}-- so we focus first on the proof the Strichartz estimate.

Before stating the Strichartz estimate, we first state Hypothesis \ref{Definition: Hypothsis H}, which will be used in the proof of the Strichartz estimate of Theorem \ref{Theorem: Short Time Strichartz}.

\begin{hypothesis}[Hypothesis $H(\alpha)$]
    Let $\alpha > 0$, $\beta\in \mathbb{R}^2$, and $\Lambda:= \mathbb{Z} + e^{i\frac{\pi}{3}}\mathbb{Z}$, the lattice spanned by $(1,0)$ and $(\frac{1}{2}, \frac{\sqrt{3}}{2}).$ We say that $\alpha$ satisfies \ref{Definition: Hypothsis H} if for all $N\in 2^\mathbb{N}$ we have
    \begin{equation}\tag{$H(\alpha)$}\label{Definition: Hypothsis H}
        \#\{n\in \Lambda\,:\,N^2\lesssim |n-\beta|^2\lesssim N^2+N^\alpha\}\lesssim N^\alpha.
    \end{equation}
    That is, the number of lattice points of $\Lambda$ in any annulus centered at $\beta$ with inner radius $N$ and outer radius $\sqrt{N^2+N^\alpha}$ is at most $N^\alpha$.    
\end{hypothesis}
\begin{remark}\label{Remark: H holds for larger}
    We observe that $\alpha > \frac23$ is clear by classical Fourier analytic techniques. Moreover, if (say) $\frac23 > \alpha' > \alpha$ and $H(\alpha)$ holds, then a decomposition argument also yields $H(\alpha')$.
\end{remark}

The above hypothesis is intimately related to the (slightly generalized) circle problem of Gauss, which is to find the best $\theta = \theta(\beta) > \frac12$ such that
\[
\#\{n\in\mathbb{Z}^2\,:\,|n-\beta|^2\leq N^2\} = \pi N^2 + O(N^\theta).
\]
Indeed, many of the methods used to prove bounds on the above are insensitive to whether or not the object is a circle or ellipse. Results of Huxley \cite{MR1044308, MR1199067, MR2005876} lead up to and obtain $\theta = \frac{131}{208} = 0.62980...$, and recently Li \& Yang \cite{li2023improvement} obtained $\theta = 0.62896...$ at least in the case that $\beta = 0$ and $\Lambda = \mathbb{Z}^2$. However, even proving that the above (for the standard lattice) holds for any $\alpha < \theta(0)$ is nontrivial. Moreover, there is a wealth of work on this kind of problem from a probabilistic perspective, see e.g. \cite{MR2039790, MR2226267, MR2253596, MR4302203, MR1246066, MR1257819, MR1224087}, viewing either the center, $\beta$, or the radius as random variables.

With \ref{Definition: Hypothsis H}, we are in a position to state our improvement to \eqref{Equation: Generic short time estimate}:
\begin{theorem}\label{Theorem: Short Time Strichartz}
    Let $\alpha > 0$ satisfy \ref{Definition: Hypothsis H}. Then for any $N\in 2^\mathbb{N}$ and $f\in L^2(\mathbb{T})$:
    \begin{equation}\label{Equation: Short time L6 thm eqn}
\|P_{\leq N} e^{it\Delta}f\|_{L^6_{x,t}(\mathbb{T}\times[0,N^{-\alpha}])}\lesssim \|P_{\leq N}f\|_{L^2(\mathbb{T})}.
    \end{equation}
    In particular, Equation \ref{Equation: Short time L6 thm eqn} holds unconditionally for $\alpha > \frac{131}{208}$.
\end{theorem}
With \ref{Definition: Hypothsis H} we can similarly restate Theorem \ref{Theorem: GWP theorem} to the form that we will actually prove.
\begin{proposition}
    Let $\alpha > 0$ satisfy \ref{Definition: Hypothsis H} and $s > \frac\alpha3$. Then \eqref{Equation: Quintic NLS} is globally well-posed in $H^s(\mathbb{T})$ for any $\|u_0\|_{L^2_{x}(\mathbb{T})}\ll 1$. Moreover, the same result holds for any $\alpha$ such that \eqref{Equation: Short time L6 thm eqn} holds.
\end{proposition}

The outline of the paper will be as follows. In Section \ref{Section: Notation} we review common notation that will be used throughout the manuscript. In Section \ref{Section: Proof of Theorem 2} we prove Theorem \ref{Theorem: Short Time Strichartz} under the assumption of \eqref{Definition: Hypothsis H}. In Section \ref{Section: Proof of prop 1} we fix an $\alpha$ such that Theorem \ref{Theorem: Short Time Strichartz} holds, and use this to prove Theorem \ref{Theorem: GWP theorem}, contingent on the $I$-method results of Section \ref{Section: I-Method Setup & Function Spaces}. We conclude the manuscripts by proving the claimed bounds of Proposition \ref{Proposition: I method bounds}: in Section \ref{Section: Atomic Space Overview} we overview the atomic spaces and construct the space $Y$; in Section \ref{Section: Basic Multilinear Estimates} we prove a trilinear Strichartz estimate on scaled torii which will serve as a replacement for the bilinear estimate of \cite{de2007global, li2011global}; lastly, Section \ref{Section: I-method Setup} contains the construction of the modified energy, with Section \ref{Section: Pointwise Symbol Estimates} containing size and regularity estimates for the terms in the modified energy, culminating in the estimation of the multilinear operators of Proposition \ref{Proposition: I method bounds} in Section \ref{Section: I-method multilinear estimates}.

\section{Notation}\label{Section: Notation}
As will be come apparent in the later sections, we will need to work on a scaled torus $\lambda\mathbb{T}$, for some $\lambda > 0$. For any function $v\in C^0_tL^2_x([0,T]\times \mathbb{T})$ we denote
\[
v^{\lambda}(x,t) = \frac{1}{\lambda^{1/2}}v\left(\frac{x}{\lambda}, \frac{t}{\lambda^2}\right),
\]
where $v^\lambda$ naturally lives in $C^0_tL^2_x([0,\lambda^2 T]\times \lambda \mathbb{T})$. We note that this respects the natural scaling of equation \eqref{Equation: Quintic NLS}.

In the scaled space $\lambda\mathbb{T}$, we denote the Fourier transform of $f$ by $\widehat{f}$, and is given by
\[
\widehat{f}(k) = \int_{\lambda\mathbb{T}}e^{-ikx}f(x)\,dx.
\]
Letting $(dk)_\lambda$ denote the normalized counting measure, we naturally have the Fourier inversion formula given by
\[
f(x) = \int e^{ikx}\widehat{f}(k)\,(dk)_\lambda = \frac{1}{\lambda}\sum_{k\in\frac{1}{\lambda}\mathbb{Z}}e^{ikx}\widehat{f}(k).
\]
If $I\subset \tfrac{1}{\lambda}\mathbb{Z}\subset \mathbb{R}$, we denote the projection of $\widehat{f}$ onto the set $I$ by $P_I f$, where
\[
P_If := \int_{I} e^{ikx}\widehat{f}(k)(dk)_\lambda = \frac{1}{\lambda}\sum_{k\in I}\widehat{f}(k)e^{ikx}.
\]
With these we find the standard definition of the scaled Sobolev spaces $H^s(\lambda\mathbb{T})$ as
\[
\|f\|^2_{H^s(\lambda\mathbb{T})} = \frac{1}{\lambda}\sum_{k\in\frac{1}{\lambda}\mathbb{Z}}\langle k\rangle^{2s}|\widehat{f}(k)|^2.
\]

We will also require the ability to go between multiple scales. Suppose that $\mu, \lambda, \nu > 0$, $v\in C^0_tL^2_x(\mu\mathbb{T}\times\lambda\mathbb{T})$, and $Q$ is some Fourier multiplier on $L^2_x(\lambda\mathbb{T})$ given by $\widehat{Q}(k) = q(k)$. Defining $Q^\nu$ to be the Fourier multiplier on $L^2(\nu\lambda\mathbb{T})$ with $\widehat{Q^\nu}(k) = q(\nu k)$, then
\begin{equation}\label{Equation: Scaling Multiplier}
\|Qf\|_{L^{p_2}_tL^{p_1}_x(\mu\mathbb{T}\times \lambda\mathbb{T})} = \nu^{-\sigma(p_1,p_2)}\|Q^{\nu}f^{\nu}\|_{L^{p_2}_tL^{p_1}_x(\nu^2\mu\mathbb{T}\times \nu\lambda\mathbb{T})}, \quad \sigma(p_1, p_2) = \frac{2}{p_2}+\frac{1}{p_1} - \frac12.
\end{equation}
The above will often be applied in the situation that $Q = P_I$ for some interval $I\subset \tfrac{1}{\lambda}\mathbb{Z}$, where setting $\nu = \lambda^{-1}$ yields
\[
\|P_If\|_{L^{p_2}_tL^{p_1}_x(\mu\mathbb{T}\times \lambda\mathbb{T})} = \lambda^{\sigma(p_1,p_2)}\|P_{\lambda I}f^{\lambda^{-1}}\|_{L^{p_2}_tL^{p_1}_x(\tfrac{\mu}{\lambda^2}\mathbb{T}\times\mathbb{T})},
\]
while setting $\nu = \mu^{-1/2}$ yields
\[
    \|P_If\|_{L^{p_2}_tL^{p_1}_x(\mu\mathbb{T}\times \lambda\mathbb{T})} = \nu^{\frac{\sigma(p_1,p_2)}{2}}\|P_{\sqrt{\mu} I}f^{\mu^{-1/2}}\|_{L^{p_2}_tL^{p_1}_x(\mathbb{T}\times\tfrac{\lambda}{\sqrt{\mu}}\mathbb{T})}.
\]
In our setting $\mu \ll \lambda^2$, so this will allow us to transfer problems on $\lambda\mathbb{T}$ to those on $\mathbb{T}$ for small time, and problems on $\mu\mathbb{T}\times\lambda\mathbb{T}$ to problems with unit time on very large torii. Additionally, we will always take $p_1=p_2\in\{2, 6\}$-- when $p_1 = p_2 = 6$, we see that lossless $L^6_{x,t}$ Strichartz for $T \sim \mu/\lambda^2$ transfer without loss to $L^6_{x,t}$ Strichartz estimates on $\mu\mathbb{T}\times \lambda\mathbb{T}$.

We will perform Littlewood-Paley decompositions, where it becomes necessary to consider the dyadic numbers $M\in 2^{\mathbb{N}_0}$, and define
\[
A_{N_i} = \{x\in\mathbb{R}\,:\,\frac{1}{2}M<|x| < 2M\}\qquad\mbox{for }M\in 2^\mathbb{N},
\]
together with $A_1 = [-2,2]$. We may then define the relationship $|k|\sim M$ in the standard way, with the understanding that $|k|\sim 2$ for $|k|\leq 2$. In this fashion, we may consider the smooth (inhomogenous) Littlewood-Paley projections $P_M$ for $M\in 2^\mathbb{N}$. For functions $f\in L^2$ we write $f_{M} := P_M f$, which affords us the standard decomposition:
\[
\|f\|_{L^2_x(\lambda\mathbb{T})}^2\sim \sum_{M\in2^\mathbb{M}}\|f_M\|_{L^2_x(\lambda\mathbb{T})}^2.
\]

For $n\in\mathbb{N}$ with $n > 1$ we also define both the hyperplanes $\Gamma_n$ and their inherited measures $(d\Gamma_n)_\lambda$ by
\[
\Gamma_n = \left\{(k_1,\dots, k_n)\in\left(\tfrac{1}{\lambda}\mathbb{Z}\right)^j\,:\,k_1+\dots+k_n = 0\right\}, \quad (d\Gamma_n)_\lambda = \prod_{j=1}^{n-1}(dk_j)_\lambda,
\]
which will be useful not only in our construction of the modified energy, but for handling multilinear estimates. Indeed, if $n > 1$ and $(f_j)_{j=1}^n$ is some collection with $f_j\in L^2(\lambda\mathbb{T})$, then we find the standard identity
\[
\int_{\lambda\mathbb{T}}\prod_{j=1}^n f_j(x)\,dx = \frac{1}{\lambda^{n-1}}\sum_{\substack{k_j\in\tfrac{1}{\lambda}\mathbb{Z}\\ \sum_{j=1}^n k_j = 0}}\prod_{j=1}^n\widehat{f}_j(k_j) = \int \prod_{j=1}^n\widehat{f}_j(k_j)\,(d\Gamma_n)_\lambda.
\]

For $n\in\mathbb{N}$ and a collection $(z_j)_{j=1}^n$, we denote $(z_j^*)_{j=1}^n$ to be a decreasing rearrangement of $(z_j)_{j=1}^n$-- that is, 
\[
|z_1^*|\geq |z_2^*|\geq\dots\geq |z_n^*|,
\]
and $(z_j^*)$ is a permutation of $(z_j)$. Lastly, for $A,B > 0$, we denote $A\vee B = \max(A, B)$ and $A\wedge B = \min(A,B).$

\section{A Lattice Counting Lemma \& Proof of Theorem \ref{Theorem: Short Time Strichartz}}\label{Section: Proof of Theorem 2}
This section is dedicated to the proof of Theorem \ref{Theorem: Short Time Strichartz}. For this section, we let $N\in 2^\mathbb{N}$ be distinct from the $N$ used in the $I-$method. 

We first perform a reduction as in \cite{herr2023strichartz}, where for fixed $n_i,m_i\in\mathbb{Z}$ with $1\leq i\leq 3$ we define 
\[
\Omega = \left|\sum_{i=1}^3n_i^2-m_i^2\right|.
\]
Additionally, for a fixed $f\in L^2_x$ and $\tau\in\mathbb{Z}$, we write
\[
h(\tau) = h(\tau; N) = \sum_{\substack{\sum_{i=1}^3n_i-m_i=0\\\Omega = \tau,\,|n_i|,|m_i|\leq 2N}}\left|\prod_{i=1}^3\widehat{f}(n_1)\widehat{f}(m_i)\right|.
\]
With these notations, we find
\begin{align*}
    \|e^{it\Delta}P_{\leq N}f\|^6_{L^6_{t,x}([0,N^{-\alpha}]\times\mathbb{T})} & = \int_0^{N^{-\alpha}}\int_\mathbb{T}|e^{it\Delta}P_{\leq N}f|^6\,dxdt\\
    &\lesssim N^{\alpha}\int_0^{2N^{-\alpha}}\int_0^{T}\int_\mathbb{T}|e^{it\Delta}P_{\leq N}f|^6\,dxdtdT\\
    &\lesssim N^{-\alpha}h(0) + \sum_{\tau > 0}\min\left(N^{-\alpha}, \frac{N^\alpha}{\tau^2}\right)h(\tau)\\
    &\lesssim N^{-\alpha}h(0) + \sup_{K\in N^\alpha2^\mathbb{N}}\frac{1}{K}\sum_{K\leq |\tau|\leq 2K}h(\tau).
\end{align*}
The first term is handled exactly the same way that one handles the $L^6_{x,t}(\mathbb{T}\times\mathbb{T})$ estimate. Indeed, $h(0)$ can be estimated exactly as in \cite[Proposition 2.36]{Bour1} yielding a loss of $N^{0+}$, and hence the first contribution satisfies
\[
N^{-\alpha}h(0)\lesssim N^{-\alpha+}\|f\|_{L^2_x(\mathbb{T})}^6\lesssim \|f\|_{L^2_x(\mathbb{T})}^6,\qquad \alpha > 0.
\]
Thus, for $\alpha > 0$ it is sufficient to instead estimate
\[
\sup_{K\in N^\alpha2^\mathbb{N}}\frac{1}{K}\sum_{K\leq |\tau|\leq 2K}h(\tau) = \sup_{K\in N^\alpha2^\mathbb{N}}\frac{1}{K}\sum_{\substack{\sum_{i=1}^3n_i-m_i=0\\K\leq |\sum_{i=1}^3n_i^2-m_i^2|\leq 2K\\|n_i|,|m_i|\leq 2N}}\left|\prod_{i=1}^3\widehat{f}(n_1)\widehat{f}(m_i)\right|.
\]
\begin{remark}
    It is known, \cite[Remark 2]{Bour1}, that some loss in estimating $h(0)$ is necessary. This forces the best possible Strichartz estimate to be of the form
    \[
        \|P_{\leq N} e^{it\Delta_\mathbb{T}}f\|_{L^6_{x,t}(\mathbb{T}\times [0,T])}\lesssim \|P_{\leq N}f\|_{L^2_{x}(\mathbb{T})},\qquad T\lesssim \log(1+N)^{-6},
    \]
    which more closely matches the $L^4_{x,t}$ estimate of \cite{herr2023strichartz} on $\mathbb{T}^2$.
\end{remark}

\begin{lemma}\label{Lemma: Dyadic component bound}
    Let $\alpha$ be such that \ref{Definition: Hypothsis H} holds, $K\in N^\alpha2^\mathbb{N}$, and $f\in L^2_x(\mathbb{T})$. Then 
    \[
\frac{1}{K}\sum_{\substack{\sum_{i=1}^3 n_i-m_i = 0\\ K\leq \left|\sum_{i=1}^3 n_1^2-m_i^2\right|\leq 2K\\ |n_i|,|m_i|\leq 2N}}|\prod_{i=1}^3\widehat{f}(n_1)\widehat{f}(m_i)| \lesssim \|P_{\leq N}f\|_{L^2_x(\mathbb{T})}^6.
    \]
\end{lemma}
Due to the form of the estimate, it becomes necessary to calculate the number of lattice points on a plane intersected with a very thin cylindrical annulus in the direction of the normal to the plane. Before getting to this, we must first introduce notation. For $n\in\mathbb{Z}$ and $K\in N^\alpha2^\mathbb{N}$ fixed, let 
\[
P_n = \{x\in\mathbb{R}^3 : x\cdot(1,1,1) = n\},
\]
which is a plane with normal vector $(1,1,1)$. For $\ell\in\mathbb{N}$, $\xi_1 = (1,0,-1)$, and $\xi_2 = \tfrac{1}{\sqrt{3}}(1,-2,1)$, we decompose $[0,6N^2] = \bigsqcup_{\ell\geq 0} I_\ell$ where $I_\ell = [\ell K, (\ell+1)K)$, $0\leq \ell\lesssim N^2/K$, and define 
\begin{multline*}
\mathcal{C}_\ell =  \{x\in \mathbb{R}^3\,:\,x = t(1,1,1),\,t\in\mathbb{R}\}\\
+ \{x\in\mathbb{R}^3\,:\,x = \xi_1 r\cos(\theta) + \xi_2 r\sin(\theta),\,\theta\in[0,2\pi],\,r^2\in I_\ell\},
\end{multline*}
which is the cylindrical annulus in the direction of the normal vector for the plane $P_n$, with inner radius $\sqrt{\ell K}$ and outer radius $\sqrt{(\ell+1)K}$. Observe that 
\[
\mathbb{Z}^3\cap P_n\cap\bigsqcup_{\ell\geq 0}\mathcal{C}_\ell = \{(n_1,n_2,n_3)\in\mathbb{Z}^3\,:\,\,n_1+n_2+n_3=n\}.
\]

\begin{lemma}\label{Lemma: Lattice count reduction}
There exists $c_n$, depending only on the equivalence class of $n$ modulo 3, such that for any $0\leq \ell\lesssim 6N^2/K$ we have:
    \[
    \#(\mathbb{Z}^3\cap\mathcal{C}_\ell\cap P_n) = \#(\mathbb{Z}^2\cap A(\ell K, (\ell+1)K) + c_n),
    \]
    where $A(r_1, r_2)$ is the elliptical annular region:
    \[
    A(r_1, r_2) := \left\{ (x,y)\in\mathbb{R}^2\,:\,r_1^2\leq (x+\tfrac{y}{2})^2+\tfrac{3}{4}y^2 < r_2^2\right\}.
    \]
    In particular, we have the relationship:
    \[
    \Lambda\cap \{n\in\mathbb{R}^2\,;\, r_1^2 \leq |n|^2\leq r_2^2\} = \mathbb{Z}^2 \cap A(r_1, r_2).
    \]
\end{lemma}
\begin{proof}
    This is almost trivial, but included for completeness. Observe that the above is a count of the number of lattice points in the affine lattice given by 
    \[
    \Gamma = \{x\,:\, x = (\lfloor\tfrac{n}{3}\rfloor,\lfloor\tfrac{n}{3}\rfloor,\lfloor\tfrac{n}{3}\rfloor + j) + a\xi_1 + \tfrac{b}{2}(\xi_1-\sqrt{3}\xi_2),\,a,b\in\mathbb{Z}\},
    \]
    for $j\equiv n\mod 3$ and $x$ satisfying 
    \[
        \ell K\leq |x-(n/3, n/3, n/3)|^2\leq (\ell+1)K.
    \]
    By a rotation and translation of $\mathbb{R}^3$, we may map $P_n$ to $\{z = 0\}$, and, after identification with $\mathbb{R}^2$ and after possibly choosing new generators,  $\Gamma$ to
    \[
    \Gamma \mapsto \{x\,:\, x = a(1,0) + b(\tfrac{1}{2},\tfrac{\sqrt{3}}{2}),\,a,b\in\sqrt{2}\mathbb{Z}\} = \sqrt{2}\Lambda,
    \]
    and $\mathcal{C}_\ell\cap P_n$ to
    \[
        \{x\in\mathbb{R}^2\,:\,x = c_n + (r\cos(\theta), r\sin(\theta)),\,\theta\in[0,2\pi],\,r^2\in 2I_\ell\},
    \]
    for some $c_n\in\mathbb{R}^2$ depending only on the equivalence class modulo $3$ of $n$. That is, we have shown
    \begin{multline*}
    \#(\mathbb{Z}^3\cap \mathcal{C}_\ell\cap P_n) = \Lambda\cap\{n\in\mathbb{R}^2\,:\,\ell K\leq |n|^2\leq (\ell+1)K\} +c_n\\
    = \{(x,y)\in\Lambda\,:\,\ell K\leq |(x,y)|^2\leq (\ell+1)K\} + c_n.
    \end{multline*}
    Now, let $L^{-1}$ denote the linear transformation that maps $(1,0)\mapsto (1,0)$ and $(0,1)\mapsto (\tfrac12, \tfrac{\sqrt{3}}{2})$. Under this, we see that $L^{-1}\mathbb{Z}^2 = \Lambda$ and hence
    \begin{multline*}
        \{(x,y)\in \Lambda\,:\,\ell K\leq |(x,y)|^2\leq (\ell+1)K\} = \{(x,y)\in\mathbb{Z}^2\,:\,\ell K\leq |L^{-1}(x,y)|^2\leq (\ell+1)K\}\\
        =\{(x,y)\in\mathbb{Z}^2\,:\,\ell K\leq (x+\tfrac{y}{2})^2+\tfrac34 y^2\leq (\ell+1)K\}.
    \end{multline*}
    It follows that, for $c_n\mapsto Lc_n$, we have the desired claim. As the second claim is merely an intermediate step, we are done.
\end{proof}
    With the problem now reduced to the number of lattice points in an elliptical annulus, we are now free to use Gauss circle problem type bounds to handle this. In particular, Huxley's work \cite[Theorem 5]{MR2005876} provides the bound
    \[
    \#\left(\mathbb{Z}^2\cap (A(0, M) + \beta)\right) = \operatorname{Vol}(A(0,1)) M^2 + O(M^{\frac{131}{208}+}).
    \]
    for any $\beta\in\mathbb{R}^2$. 
    Standard differencing then yields that
    \[
        \#\left(\mathbb{Z}^2\cap (A(\ell K, (\ell+1)K) + \beta)\right) = \operatorname{Vol}(A(0,1)) K + O\left((\ell K)^{\frac{131}{416}+}\right)
         = \operatorname{Vol}(A(0,1)) K + O(N^{\frac{131}{208}+}),
    \]
    which, for $\alpha > \frac{131}{208}$, is our unconditional statement. More generally, suppose that $\alpha > 0$ satisfies \ref{Definition: Hypothsis H} and we assume that $K\gtrsim N^\alpha$. Then $\ell K \lesssim N^2$ so that Remark \ref{Remark: H holds for larger} and the second portion of Lemma \ref{Lemma: Lattice count reduction} yields
    \[
        \mathbb{Z}^2\cap (A(\ell K, (\ell+1)K) + \beta) = \Lambda\cap \{n\in\mathbb{Z}^2\,;\, \ell K \leq |n-\tilde{\beta}|^2\leq (\ell+1)K\}\lesssim K.
    \]
    In particular, this yields the following Corollary.
\begin{corollary}\label{Corollary: Annular Lattice Point Count}
    Let $\alpha > 0$ so that \ref{Definition: Hypothsis H} holds, $N^\alpha \lesssim K$, $\ell\lesssim N^2/K$, and $c\in\mathbb{R}^2$. Then
    \[
    \#(\mathbb{Z}^2\cap (A(\ell K,(\ell+1)K) + c))\lesssim K.
    \]
\end{corollary}

We now have all the pieces necessary to prove the Lemma \ref{Lemma: Dyadic component bound}.
\begin{proof}[Proof of Lemma~\ref{Lemma: Dyadic component bound}]
    Recall that we seek to show:
    \[
    \frac{1}{K}\sum_{\substack{\sum_{i=1}^3 m_i-\tilde{m}_i = 0\\ K\leq \left|\sum_{i=1}^3 m_1^2-\tilde{m}_i^2\right|\leq 2K\\ |m_i|,|\tilde{m}_i|\leq N}}|\prod_{i=1}^3\widehat{f}(m_1)\widehat{f}(\tilde{m}_i)| \lesssim \|P_{\leq N}f\|_{L^2_x(\mathbb{T})}^6.
    \]
    Now, we first foliate over the linear relationship, and write
    \[
    \frac{1}{K}\sum_{\substack{\sum_{i=1}^3 m_i-\tilde{m}_i = 0\\ K\leq \left|\sum_{i=1}^3 m_1^2-\tilde{m}_i^2\right|\leq 2K\\ |m_i|,|\tilde{m}_i|\leq N}}|\prod_{i=1}^3\widehat{f}(m_1)\widehat{f}(\tilde{m}_i)| \lesssim \frac{1}{K}\sum_{n\in\mathbb{Z}}\sum_{\substack{m,\tilde{m}\in P_n\\K\leq| |m|^2-|\tilde{m}|^2|\leq 2K\\ |m|,|\tilde{m}|\leq 3N}}|\tilde{f}(m)\tilde{f}(\tilde{m})|,
    \]
    where, for $m\in\mathbb{Z}^3$:
    \[
    \tilde{f}(m) = \widehat{f}(m_1)\widehat{f}(m_2)\widehat{f}(m_3).
    \]
    We now introduce the decomposition into annuli as above. In order to do this, we decompose $\mathbb{R}^3 = \sqcup_{\ell\geq 0}\mathcal{C}_\ell$, and denote $\mathcal{C}^n_\ell=P_n\cap\mathcal{C}_\ell$. By almost-orthogonality in the $\mathcal{C}_\ell$ decomposition, we now write our estimate as:
    \begin{multline*}
        \frac{1}{K}\sum_{n\in\mathbb{Z}}\sum_{\substack{m,\tilde{m}\in P_n\\K\leq | |m|^2-|\tilde{m}|^2|\leq 2K\\ |m|,|\tilde{m}|\leq 3N}}|\tilde{f}(m)\tilde{f}(\tilde{m})|\
        \lesssim \frac{1}{K}\sum_{n\in\mathbb{Z}}\sum_{\substack{\ell,j\geq 0\\|l-j|=O(1)}}\|\tilde{f}\|_{\ell^1(\mathcal{C}_\ell^n\cap B(0,10N))}\|\tilde{f}\|_{\ell^1(\mathcal{C}_j^n\cap B(0,10N))},
    \end{multline*}
    where Cauchy-Schwarz and Lemma~\ref{Lemma: Lattice count reduction} allows us to write
    \begin{multline*}
    \|\tilde{f}\|_{\ell^1(\mathcal{C}_\ell^n\cap B(0,10N))}\lesssim \#(\mathbb{Z}^3\cap\mathcal{C}_\ell^n\cap B(0,10N))^\frac{1}{2}\|\tilde{f}\|_{\ell^2(\mathcal{C}_\ell^n\cap B(0,10N))}\\
    \lesssim \#(\mathbb{Z}^2\cap \left(A(\ell K, (\ell+1)K)+c\right)\cap B(0,10N))^\frac{1}{2}\|\tilde{f}\|_{\ell^2(\mathcal{C}_\ell^n\cap B(0,10N))}\\
    \lesssim K^\frac{1}{2}\|\tilde{f}\|_{\ell^2(\mathcal{C}_\ell^n\cap B(0,10N))},
    \end{multline*}
    for some $c$ depending solely on the equivalence class of $n$ modulo $3$. Summarizing, we have found:
    \begin{equation*}
        \frac{1}{K}\sum_{n\in\mathbb{Z}}\sum_{\substack{m,\tilde{m}\in P_n\\K\leq | |m|^2-|\tilde{m}|^2|\leq 2K\\ |m|,|\tilde{m}|\leq 3N}}|\tilde{f}(m)\tilde{f}(\tilde{m})|
        \lesssim 
        \sum_{n\in\mathbb{Z}}\sum_{\substack{\ell,j\geq 0\\|l-j|=O(1)}}\|\tilde{f}\|_{\ell^2(\mathcal{C}_\ell^n\cap B(0,10N))}\|\tilde{f}\|_{\ell^2(\mathcal{C}_j^n\cap B(0,10N))}.
    \end{equation*}
    Another application of Cauchy-Schwarz yields:
    \[
\sum_{n}\sum_{\substack{\ell,j\geq 0\\|l-j|=O(1)}}\|\tilde{f}\|_{\ell^2(\mathcal{C}_\ell^n\cap B(0,10N))}\|\tilde{f}\|_{\ell^2(\mathcal{C}_j^n\cap B(0,10N))}\lesssim \sum_n\|\tilde{f}\|_{\ell^2(P_n)}^2\lesssim \|f\|_{L^2}^6,
    \]
    as desired.
\end{proof}

\section{Global well-posedness of Quintic NLS}\label{Section: Proof of prop 1}
We now return to the discussion of \eqref{Equation: Quintic NLS}, and, for some large $N$, recall of the definition of the $m$ multiplier in equation \eqref{Definition: m definition}. For the remainder of the manuscript, we fix
\[
\lambda \sim N^{\frac{1-s}{s}+}.
\]
and let $\alpha > 0$ be such that \eqref{Equation: Short time L6 thm eqn} holds. We remark that this occurs whenever $\alpha$ satisfies \ref{Definition: Hypothsis H}, but the converse may not be true-- the remainder of the manuscript only requires that \eqref{Equation: Short time L6 thm eqn} holds.

With these conventions, we let $u$ be a solution to \eqref{Equation: Quintic NLS}, scale by $\lambda$, and then apply the $I$ operator. In this fashion, we observe
\[
\|Iu^\lambda\|_{\dot{H}^1_x}\lesssim N^{1-s}\|u^\lambda\|_{\dot{H}^s_x}\sim \frac{N^{1-s}}{\lambda^s}\|u\|_{\dot{H}^s_x}\lesssim N^{0-}\|u\|_{\dot{H}^s_x},
\]
and hence 
\begin{equation}\label{Equation: I homogenous bound}
\|Iu_0^\lambda\|_{H^1_x}\lesssim \|u_0\|_{L^2_x} + N^{0-}\|u_0\|_{\dot{H}^s_x}\lesssim \|u_0\|_{L^2_x},
\end{equation}
for $N\gg 1$.

Furthermore, we observe that if $u^\lambda$ is a solution on $\lambda\mathbb{T}$ for some time $T$, then $u$ solves \eqref{Equation: Quintic NLS} on $\mathbb{T}$ for time $T/\lambda^2$. Additionally, if $u^\lambda$ is a solution for time $T$, then $Iu^\lambda$ solves the $I-$system, given by:
\begin{equation}\label{Equation: NLS with I operator}
    \begin{cases}
    (i\partial_t+\partial_{x}^2)Iu &= \pm I(|u|^4u),\qquad(t,x)\in[0, T]\times\lambda\mathbb{T}\\
    Iu(x,0) &= Iu_0\in H^1(\lambda\mathbb{T}).
    \end{cases}
\end{equation}
While $Iu^\lambda\in H^1$, it does not necessarily have conservation of energy. However, by the Gagliardo-Nirenberg inequality and \eqref{Equation: I homogenous bound} we at least have the existence of a constant $C_1$ such that
\begin{equation}\label{Equation: I method base bound}
\|Iu^\lambda(t)\|_{H^1_x}^2< C_1(E^1_I(u^\lambda)(t) + \|u_0\|_{L^2_x}^2)
\end{equation}
for $\|u_0\|_{L^2_x}\ll 1$. This realization, together with the nice cancellation structure within $E$, motivates the study of $E$ to control the increment of the energy after each new local solution. 

We observe that, even in the presence of the $I-$operator, we may solve \eqref{Equation: NLS with I operator} and have a standard local well-posedness bound. This will enable us to estimate the energy in the $Y^1_T$ space so long as $T\lesssim \lambda^{2-\alpha}N^{-\alpha}$, and is the content of the next Lemma.

\begin{lemma}[Proposition 3.3, \cite{schippa2024improved}]\label{Lemma: LWP}
    Let $0 < s < 1$, $0 < T\sim \lambda^{2-\alpha}N^{-\alpha}$, and $I$ the $I-$operator defined as a Fourier multiplier with multiplier \eqref{Definition: m definition}. There exists a $0 < C\ll 1$ such that if $\|Iu_0\|_{H^1(\lambda\mathbb{T})}< C$, then \eqref{Equation: NLS with I operator} is locally well-posed in $Y^1_T$ (see Section \ref{Section: Atomic Space Overview}), and $\|Iu\|_{Y^1_T}\leq 2C$.
\end{lemma}
\begin{remark}
The above lemma is the source of the small data assumption on Theorem \ref{Theorem: GWP theorem}: we require small data in order to close the contraction, see \cite[Proposition 3.1]{schippa2024improved}. On the other hand, \eqref{Equation: I homogenous bound} allows us to instead control the $L^2$ norm.
\end{remark}
In Section \ref{Section: I-method Setup} we prove the following proposition, which is a  modification of the result of \cite{schippa2024improved}, allowing for longer times. The exact definition of $\Lambda_k$, $\tilde{\sigma}_6$, $\overline{M}_6$, and $M_{10}$ don't matter for this proof, and are constructed later on; similarly, $Y_T^s$ is reviewed in Section \ref{Section: Atomic Space Overview}.
\begin{proposition}\label{Proposition: I method bounds}
    Let $t > 0$, $s > 0$, $\lambda \gg 1$, and $u^\lambda$ the solution to \eqref{Equation: Quintic NLS} emanating from $u_0^\lambda\in H^s(\lambda\mathbb{T})$. There are symbols $\tilde{\sigma_6},\, \overline{M}_6:\mathbb{R}^6\to\mathbb{R}$, and $M_{10}:\mathbb{R}^{10}\to\mathbb{R}$ such that
    \begin{equation}\label{Equation: I method FTC}
E^1_I(u^\lambda)(t) = E^1_I(u^\lambda)(0) + [\Lambda_6(\tilde{\sigma}_6)]_{s=0}^t + \int_0^t \Lambda_6(\overline{M}_6) + \Lambda_{10}(M_{10})\,ds.
    \end{equation}
    Moreover, for $0 < t \lesssim \lambda^{2-\alpha}N^{-\alpha}$ there is a $\delta(s) > 0$ such that these terms satisfy the following bounds:
    \begin{align*}
        \left|\Lambda_6(\tilde{\sigma}_6:u^\lambda,\dots,u^\lambda)(t)\right|&\lesssim N^{-\delta(s)}\|Iu^\lambda(t)\|_{H^1(\lambda\mathbb{T})}^6,\\
        \left|\int_0^t\Lambda_6(\overline{M}_6;u^\lambda,\dots,u^\lambda)(s)\,ds\right|&\lesssim N^{-3+}\|Iu^\lambda\|_{Y^1_T}^6,\\
        \left|\int_0^t\Lambda_{10}(M_{10};u^\lambda,\dots,u^\lambda)(s)\,ds\right|&\lesssim N^{-3+}\|Iu^\lambda\|_{Y^1_T}^{10}.
    \end{align*}
\end{proposition}
The above proposition can be viewed as the combination of \cite[Proposition 5.7, 5.10, 5.11]{schippa2024improved}. However, the proofs of \cite{schippa2024improved} required $s > \frac13$, and going beyond $\alpha = 1$ necessitates the introduction of a trilinear Strichartz estimate (see Section \ref{Section: Basic Multilinear Estimates}). 


With this proposition, we are in a position to prove Theorem \ref{Theorem: GWP theorem}. 
\begin{proof}[Proof of Theorem \ref{Theorem: GWP theorem}]
    This is a standard continuity argument, whose details are included for completeness. Let $s > \frac{\alpha}{3}$, $u_0\in H^s(\mathbb{T})$, and
    \[
    B(D) = \{u_0\in L^2(\mathbb{T})\,:\, \|u_0\|_{L^2(\mathbb{T})}< D\}.
    \]
    We seek to show that for any $T > 0$, $D\ll 1$, and $u_0\in B(D)$, there is an $N = N(T, \|u_0\|_{H^s_x})$ so that we may solve the associated $\lambda \mathbb{T}$ problem to time $\lambda^2 T$.

    To that end, we let $T > 0$ and $\lambda \sim N^{\frac{1-s}{s}+}\gg 1$, and observe by \eqref{Equation: I method base bound} we have
    \[
    \|Iu_0^{\lambda}\|_{H^1(\lambda\mathbb{T})}^2\leq C_1(E^1_I(u^\lambda)(0) + \|u_0\|_{L^2(\mathbb{T})}^2)\leq C_2\|u_0\|_{L^2(\mathbb{T})}^2
    \]
    for $N = N(\|u_0\|_{H^s_x})\gg 1$, where $C_1$ was uniform and $C_2$ can be made uniform for $\|u_0\|_{L^2_x} < D$. It follows that for $D\ll 1$, we may make $C_2D^2 < C^2$ small enough to apply Lemma \ref{Lemma: LWP}.
    
    Now, for a fixed $u_0\in B(D)$ and some $\sqrt{C_2} D < \tilde{C} < C$ as in Lemma \ref{Lemma: LWP}, we consider the set
    \[
    A_T := \{s\in [0, \lambda^2 T] \, : \, \|Iu^{\lambda}(t)\|_{H^1(\lambda\mathbb{T})}^2 \leq \tilde{C}\mbox{ for } t\in [0,s]\}.
    \]
    By the preceding remark, this set is non-empty. Additionally, we have by continuity that this set is closed, so it suffices to show that this set is also open. To that end, let $t\in A_T$, and observe that for all $s \in [0,t]\subset [0, \lambda^2 T]$ we have $\|Iu^\lambda\|_{H^1(\lambda\mathbb{T})}\leq \tilde{C} < C$, and hence repeated applications of Lemma \ref{Lemma: LWP} implies
    \begin{multline}
        \left|\int_0^t \Lambda_6(\overline{M}_6) + \Lambda_{10}(\overline{M}_{10})\,ds\right|\leq \sum_{\substack{\sqcup I = [0,t]\\ |I|\sim \lambda^{2-\alpha}N^{-\alpha}}}\left|\int_I \Lambda_6(\overline{M}_6) + \Lambda_{10}(M_{10})\,ds\right|\\
        \leq C_3\sum_{\substack{\sqcup I = [0,t]\\ |I|\sim \lambda^{2-\alpha}N^{-\alpha}}}N^{-3+}\tilde{C}^6\leq C_3\tilde{C}^6N^{-3+\alpha+}\frac{t}{\lambda^{\alpha-2}}\leq C_3\tilde{C}^6N^{-3+\alpha+\alpha\frac{1-s}{s}}T,
    \end{multline}
    where $C_3$ is uniform and 
    \[
-3+\alpha+\alpha\frac{1-s}{s} = -3+\frac{\alpha}{s}.
    \]
    It follows that we have the bound
    \[
    \left|\int_0^t \Lambda_6(\overline{M}_6) + \Lambda_{10}(M_{10})\,ds\right|\leq N^{0-}T,
    \]
    so long as $s > \frac{\alpha}{3}$. Combining the above display with \eqref{Equation: I method base bound} and \eqref{Equation: I method FTC}, we find
    \begin{multline*}
        \|Iu^\lambda(t)\|_{H^1(\lambda\mathbb{T})}^2\leq C_1(\|u_0\|_{L^2(\mathbb{T})}^2 + E^1_I(u^{\lambda})(t))\\
        \leq C_1\|u_0\|_{L^2(\mathbb{T})}^2 + C_1|E^1_I(u^{\lambda})(0)| + C_1\left|[\Lambda_6(\tilde{\sigma}_6)]_{s=0}^t\right| + C_1\left|\int_0^t \Lambda_6(\overline{M}_6) + \Lambda_{10}(M_{10})\,ds\right|\\
        \leq C_2\|u_0\|_{L^2(\mathbb{T})}^2 + C_1N^{0-}T\leq C_2D^2 + C_1 N^{0-}T.
    \end{multline*}
    Hence, for $N = N(\|u_0\|_{H^s_x}, T)\gg 1$ we can iterate Lemma \ref{Lemma: LWP} again. 
    
    The above argument proves that $A_T$ is open for $N = N(\|u_0\|_{H^s_x}, T)$ large enough, and hence $A_T = [0,\lambda^2T]$. Unscaling yields the existence of $u$ on $[0,T]$, and thus for any $T > 0$ the solution emanating from $u_0\in B(D)$ extends to a solution on $[0,T]$. As both $u\in B(D)$ and $T$ were arbitrary, we have shown the global existence of \eqref{Equation: Quintic NLS} for $u_0\in B(D)$.
\end{proof}
\section{Function Spaces \& Multilinear Strichartz Estimates}\label{Section: I-Method Setup & Function Spaces}
The remainder of the manuscript is dedicated to the proof of Proposition \ref{Proposition: I method bounds}. In that direction, we will need to first define the space that will enable us to obtain a local well-posedness lemma, Lemma \ref{Lemma: LWP}, from our short time Strichartz estimate of Theorem \ref{Theorem: Short Time Strichartz}. Moreover, we will need to review basic interpolation results for these spaces so that we will be able to transfer the trilinear Strichartz estimate of Proposition \ref{Proposition: Trilinear Free Estimate} to a trilinear estimate in our space. With these tools under our belt, we will be able to construct the modified energy of \cite{schippa2024improved} and prove Lemmas \ref{Lemma: Pointwise tilde sigma 6}, \ref{Lemma: 10 linear operator}, and \ref{Lemma: M6 integral term}.

\subsection{Atomic Space Overview}\label{Section: Atomic Space Overview}

We will perform our estimates in the atomic spaces of \cite{MR2824485} (see also \cite{MR3200335, MR2094851}), as these behave well in critical scenarios and preserve the embeddings we enjoy. However, these spaces don't behave as well as the $L^2_{x,t}$ based spaces that we normally work with. We may still perform Littlewood-Paley decompositions, but we may not do the crude work we did before when we had that the norms were in terms of the absolute values of the space-time Fourier transforms.

We briefly review their properties below.
\begin{definition}
    Let $X$ be a separable HIlbert space over $\mathbb{C}$, and $\mathcal{Z}$ be the collection of all partitions $-\infty < t_0 < t_1 < \cdots < t_k\leq \infty$. Then for $1\leq p < \infty$ we define $V^pX$ as the space of all right-continuus functions $u:\mathbb{R}\to X$ with $\lim_{t\to-\infty} u(t) = 0$ satisfying
    \[
\|u\|_{V^pX}^p := \sup_{\{t_k\}_{k=0}^K}\sum_{k=1}^K\|u(t_k)-u(t_{k-1})\|_{X}^p < \infty.
    \]
    Additionally, we define the space $V^p_{\Delta}$ as $V^pL^2$ adapted to the linear group $e^{it\Delta}$ as
    \[
    V^p_{\Delta} = \{u\,:\,\mathbb{R}\times \lambda\mathbb{T}\to \mathbb{C}\,:e^{-it\Delta}u\in V^pL^2(\lambda\mathbb{T})\}.
    \]
\end{definition}
$V^p$ spaces also admit the nice predual $U^p$. 
\begin{definition}
    A $U^p$ atom is a function $u:\mathbb{R}\times\lambda\mathbb{T}\to \mathbb{C}$ such that
    \[
u = \sum_{k=1}^K\chi_{[t_{k-1},t_k)}a_k,\quad \sum_{k=1}^K\|a_k\|_{L^2}^p = 1,\quad\{t_k\}_{k=0}^K\in\mathcal{Z}.
    \]
    We then define the $U^p$ norm of $u:\mathbb{R}\times\lambda \mathbb{T}\to \mathbb{C}$ as
    \[
\|u\|_{U^p} := \inf\left(\left\{\sum_{k=1}^\infty |\mu_k|\,:\,u = \sum_{k=1}^\infty \mu_k f_k,\,\,\{\mu_k\}\in\ell^1,\,f_k\mbox{ a }U^p-\mbox{atom}\right\}\right).
    \]
    We similarly define the standard adaptation of $U^p$ to the linear group as $U^p_\Delta$.
\end{definition}
\begin{remark}\label{Remark: Embeddings}
    This definition of $V^p$ is commonly denoted as $V^p_{-,rc}$, \cite{MR2526409}. With this convention it follows that for all $1\leq p<q<\infty$ the embeddings $U^p\subset U^q$, $V^p\subset V^q$, $U^p\subset V^p$, and $V^p\subset U^q$ are all continuous.
\end{remark}

We define the space that we will work in as $Y^s$, which is given in terms of $V^p_\Delta$ as:
\begin{definition}
    Let $s\in\mathbb{R}$. We define $Y^s_{\pm}$ to be the space of functions $u$ for which:
    \[
\|u\|_{Y^s_{\pm}}^2 := \frac{1}{\lambda}\sum_{k\in\frac{1}{\lambda}\mathbb{Z}}\langle k\rangle^{2s}\left\|e^{\pm it|k|^2}\widehat{u}(t,k)\right\|_{V^2}^2 < \infty.
    \]
    Similarly, we define the the standard adaptation to $t\in[0,T]$ as $Y^s_{\pm,T}$.
\end{definition}

We ignore the presence in the conjugates, and simply call these spaces $Y^s$. These spaces enjoy the properties that we demand with respect to the Duhamel operator, as well as both preservation of disjointness and transference of Strichartz estimates. This is summarized in the next proposition.
\begin{proposition}[\cite{schippa2024improved, MR2526409}]\label{Proposition: Randoma atomic space facts}
    Let $s\in\mathbb{R}$ and $T > 0$. If $A$ and $B$ be disjoint subsets of $\lambda\mathbb{Z}$, then
    \[
    \|P_{A\cup B}u\|_{Y^s}^2 = \|P_A u\|_{Y^s}^2 + \|P_B u\|_{Y^s}^2.
    \]
    Moreover, if $f\in H^s(\lambda\mathbb{T})$ then 
    \[
    \|\chi_{[0,T)}(t)e^{it\Delta}f\|_{Y^s}\sim \|f\|_{H^s(\lambda\mathbb{T})},
    \]
    and if $\|e^{it\Delta}f\|_{L^p_{x,t}}\lesssim \|f\|_{L^2_x}$ for all $f\in L^2$ then
    \[
\|u\|_{L^p_{x,t}}\lesssim \|u\|_{U^p_{\Delta}}.
    \]
    
    Lastly, if $F\in L^1_tH^s_x$ then
    \[
    \left\|\chi_{[0,T)}(t)\int_0^te^{i(t-s)\Delta}F(s)\,ds\right\|_{Y^s}\lesssim \sup_{\substack{v\in Y^{-s}\\\|v\|_{Y^{-s}}\leq 1}}\left|\int_0^T\int_{\lambda\mathbb{T}}\overline{v}F(s)\,dxds\right|.
    \]
\end{proposition}
These spaces enjoy nice interpolation properties, which is the primary way we will obtain multilinear estimates in these spaces.
\begin{proposition}[Propositions 2.19 \& 2.20, \cite{MR2526409}]\label{Proposition: Interpolation Properties}
    Let $T_0:L^2\times\dots\times L^2\mapsto L^1_{\mathrm{loc}}(\mathbb{T};\mathbb{C})$ be an $n-$linear operator and $1\leq p,q \leq \infty$. If
    \[
\|T_0(e^{\cdot i\Delta}\phi_1,\dots,e^{\cdot i\Delta}\phi_n)\|_{L^p_t(\mathbb{R};L^q_x(\mathbb{T}))}\lesssim \prod_{j=1}^n\|\phi_j\|_{L^2},
    \]
    then there is an extension $T:U^p_\Delta\times\dots\times U^p_\Delta\to L^p_t(\mathbb{R};L^q_x(\mathbb{T}))$ such that $T(u_1,\dots u_n)(t)(x) = T_0(u_1(t),\dots, u_n(t))(x)$ almost everywhere and 
    \[
\|T(u_1,\dots,u_n)\|_{L^p_t(\mathbb{R};L^q_x(\mathbb{T}))}\lesssim \prod_{j=1}^n\|u_j\|_{U^p_\Delta}.
    \]
    Suppose now that $1\leq 2 < p$, $j \geq 2$, $E$ is a Banach space, $T:(U^p_\Delta)^j\to E$ is a bounded linear operator with $\|T(\{u_k\})\|_{E}\leq C_p\prod_{k=1}^j\|u_k\|_{U^p_\Delta}$ for all $\{u_k\}\in (U^p_\Delta)^j$, and that there is some $C_2\in (0, C_p]$ such that $\|T(\{u_k\})\|_{E}\leq C_2\prod_{k=1}^j\|u_k\|_{U^2_\Delta}$ holds for all $\{u_k\}\in (U^2_\Delta)^j$. Then $T$ also satisfies
    \[
\|T(\{u_k\})\|_{E}\lesssim_{p}C_2\left(\ln\frac{C_p}{C_2} + 1\right)^j\prod_{k=1}^j\|u_k\|_{V^2_\Delta},\qquad \{u_k\}\in (V^p_\Delta)^j.
    \]
\end{proposition}
The first part of the above proposition is the analogue of the usual transference principle utilized for Bourgain spaces. The second enables us to upgrade multilinear estimates for operators, $T$, that are bounded for functions in $U^p$, to estimates of $T$ considered as operators from $V^2_\Delta$, which is desirable due to our Strichartz embedding properties for $U^p_\Delta$. These above results trivially extend to $Y^s$.

Before concluding, we also remark the following by Bernstein:
\[
\|P_N u\|_{L^\infty_{x,t}([0,T]\times\lambda\mathbb{T})}\lesssim N^{\frac12}\|P_Nu\|_{L^\infty_tL^2_x([0,T]\times\lambda\mathbb{T})}\lesssim N^{\frac12}\|P_Nu\|_{U^p_\Delta}\lesssim N^{\frac12}\|P_Nu\|_{V^2_\Delta}\sim \|P_Nu\|_{Y^{\frac12}},
\]
for any, say, $2 < p < \infty$. This will be utilized, together with the Strichartz embedding, to yield estimates in $Y^s_{T}$.
\subsection{Basic Multilinear Estimates}\label{Section: Basic Multilinear Estimates}
For the remainder of the manuscript we will work on $\lambda\mathbb{T}$ at time scale $T\sim \lambda^{2-\alpha}N^{-\alpha}$. To ease exposition, we define 
\begin{equation}\label{Equation: mu definition}
    \mu := \lambda^{2-\alpha}N^{-\alpha},
\end{equation}
so that our multilinear estimates will instead be over $\mu\mathbb{T}\times\lambda\mathbb{T}$.

Before proceeding, we state the estimate of Theorem \ref{Theorem: Short Time Strichartz} in this setting, which was the original motivation for the study of short time estimates in connection with global well-posedness. 
\begin{proposition}\label{Proposition: Scaled Strichartz}
    If $\phi\in L^2(\lambda\mathbb{T})$, $N\gg 1$, $\lambda\sim N^\frac{1-s}{s}\gg 1$, and $M\in 2^\mathbb{N}$, then we find the following Strichartz estimate:
    \begin{align*}
        \|P_{M}e^{it\Delta}\phi\|_{L^6(\mu\mathbb{T}\times\lambda\mathbb{T})}\lesssim 
        \begin{cases}
            \|P_M\phi\|_{L^2(\lambda\mathbb{T})} & M\lesssim N\\
            (\lambda M)^{0+}\|P_M\phi\|_{L^2(\lambda\mathbb{T})} & M\gg N.
        \end{cases}
    \end{align*}
\end{proposition}
\begin{proof}
    This is a straightforward scaling argument. Observe that
    \[
\|P_{M}e^{it\Delta}\phi\|_{L^6(\mu\mathbb{T}\times\lambda\mathbb{T})} = \|P_{\lambda M}e^{it\Delta}\phi^{\lambda^{-1}}\|_{L^6((\lambda N)^{-\alpha}\mathbb{T}\times\mathbb{T})},
    \]
    so that if $M\lesssim N$ then we are in a position to apply Theorem \ref{Theorem: Short Time Strichartz}, while if $M\gg N$ we apply the standard $L^6$ estimate:
    \begin{multline*}
    \|P_Me^{it\Delta}\phi\|_{L^6_{x,t}(\mu\mathbb{T}\times\lambda\mathbb{T})} \lesssim \|P_{\lambda M} e^{it\Delta}\phi^{\lambda^{-1}}\|_{L^6_{x,t}(\mathbb{T}\times\mathbb{T})}\\
    \lesssim (\lambda M)^{0+}\|P_{\lambda M}\phi^{\lambda^{-1}}\|_{L^2_x(\mathbb{T})} = (\lambda M)^{0+}\|P_{M}\phi\|_{L^2_x(\mathbb{\lambda T})}.
    \end{multline*}
\end{proof}

Due to the above transference principle, this space preserves multilinear Strichartz estimates, such as those in \cite{de2007global}. In that direction, we first state the scale $T\sim 1$ bilinear Strichartz estimates given on $\lambda\mathbb{T}$ of \cite{li2011global}.

\begin{proposition}[Proposition 2.1 \& Corollary 2.1, \cite{li2011global}]\label{Proposition: Generic bilinear estimates}
    Let $M \gg 1$ and
    \[
    \widehat{I^{\pm}_M(f,g)}(k) = \int_{k=k_1+k_2}\chi_{|k_1\mp k_2|\gtrsim M}\widehat{f}(k_1)\widehat{g}(\pm k_2)(dk_1)_\lambda.
    \]
    Then for all $\phi_1,\phi_2\in L^2(\lambda\mathbb{T})$ we have the following bilinear Strichartz estimate:
    \[
    \|I^{\pm}_M(e^{it\Delta}\phi_1,e^{it\Delta}\phi_2)\|_{L^2_{t,x}(\mathbb{T}\times\lambda\mathbb{T})}\lesssim \left( \frac{1}{M}+\frac{1}{\lambda}\right)^{\frac12}\|\phi_1\|_{L^2(\lambda \mathbb{T})}\|\phi_2\|_{L^2(\lambda\mathbb{T})}.
    \]
\end{proposition}

We now observe that if $\mu$ is as in \eqref{Equation: mu definition}, then
\begin{multline*}
\|I^{\pm}_M(e^{it\Delta}\phi_1,e^{it\Delta}\phi_2)\|_{L^2_{t,x}(\mu\mathbb{T}\times \lambda\mathbb{T})} = \mu^\frac14\|I^{\pm}_{\sqrt{\mu}M}(e^{it\Delta}\phi_1^{\sqrt{\mu}^{-1}},e^{it\Delta}\phi_2^{\sqrt{\mu}^{-1}})\|_{L^2_{t,x}(\mathbb{T}\times \tfrac{\lambda}{\sqrt{\mu}}\mathbb{T})}\\
\lesssim \left(\frac{1}{M}+\frac{\mu}{\lambda}\right)^\frac12 \|\phi_1\|_{L^2_x(\lambda\mathbb{T})}\|\phi_2\|_{L^2_x(\lambda\mathbb{T})}.
\end{multline*}
This would be acceptable only if $\alpha\geq 1$, which is not the case. However, the key issue with the above is that we lose a factor of $\mu^\frac14$, which we cannot overcome, and is due to the fact that the $L^4$ norm is not scale invariant. We overcome this issue with a weaker estimate in the form of a trilinear $L^2$ estimate-- an estimate that no longer loses an extra $\mu^\frac14$ factor.

The use of trilinear estimates in this exact setting goes back to Bourgain \cite[Lemma 4.1]{MR2124457}, where he proved that \eqref{Equation: Quintic NLS} was globally well-posed for \textit{some} $s < \frac12$. Schippa \cite[Theorem 1.2]{Schippa2024} then quantified the result of Bourgain, proving:
\begin{theorem}[Theorem 1.2, \cite{Schippa2024}]\label{Theorem: Bourgain Trilinear}
Let $0 < \alpha \leq 1$ and $N_j\in 2^\mathbb{N}$ $(1\leq j\leq 3)$ be large. Suppose that $N_1\geq N_2\geq N_3$ and for some $1 \geq \beta > 0$:
\[
N_1 > N_3^{1+\beta}.
\]
Then for any $\varepsilon > 0$ and $f_j$ $(1\leq j\leq 3)$ with $\mathrm{supp }\widehat{f}_j\subset [N_j, 2N_j]$ we have:
\[
\int_{\mathbb{T}\times [0, N_1^{-\alpha}]}\prod_{j=1}^3 |e^{it\Delta}f_j|^2\,dxdt\lesssim_{\varepsilon, \beta}\log(N_1)^{12+\varepsilon}N_1^{-\frac{\alpha\beta}{8}}\prod_{j=1}^3\|f_j\|_{L^2_{x}}^2.
\]
\end{theorem} 

While one could certainly utilize Theorem \ref{Theorem: Bourgain Trilinear}, one would not be able to obtain the full range $s > \alpha/3$. Instead, we opt to prove a different trilinear Strichartz estimate that, while worse with respect to losses, will be better in terms of gain in $N$. The guiding idea is that the trilinear $L^2$ estimate on $\tfrac{1}{\lambda}\mathbb{T}$ should look very similar to the bilinear $L^2$ estimate on $(\tfrac{1}{\lambda}\mathbb{T})^2$. Indeed, the argument of the proof of Proposition  
\ref{Proposition: Trilinear Free Estimate} follows closely the argument of \cite[Proposition 4.6]{de2007global}.

\begin{proposition}\label{Proposition: Trilinear Free Estimate}
    Let $\lambda\gg 1$, $I_1, I_2, I_3\subset \frac{1}{\lambda}\mathbb{Z}$ be dyadic intervals such that $|I_1|\lesssim |I_2|\lesssim |I_3|$. If $N_{\max}$ is defined by
    \[
    N_{\max} = \max(N_{13}, N_{23})\gg 1, \quad N_{ij} = \mathrm{dist}(I_i,I_j),\quad 1\leq i\leq  j\leq 3,
    \]
    and satisfies $N_{\max}\gg |I_2|$, then for any $\phi_j\in L^2(\mathbb{T})$, $1\leq j\leq 3$, with $\mathrm{supp}\,\widehat{\phi}_j\subset I_j$ we have the trilinear estimate:
    \begin{equation*}\label{Equation: Trilinear base}
\|\prod_{j=1}^3e^{it\Delta}\phi_j\|_{L^2_{x,t}(\lambda\mathbb{T}\times\mathbb{T})}\lesssim \left(\frac{1}{\lambda} + \frac{|I_2|}{N_{\max}}\right)^\frac{1}{2}\prod_{j=1}^3\|\phi_j\|_{L^2_x(\lambda\mathbb{T})}.
    \end{equation*}
    Alternatively, if $N_{\max}\gtrsim |I_2|$ and $|I_1|\ll|I_2|$, $I_1-I_3\subset B(0, J)$ for some $J > 0$, and
 \begin{equation}\label{Condition: M condition}
M := \frac{|I_1|(J+|I_1|)}{N_{23}}\ll |I_2|\wedge N_{23},
 \end{equation}
    then we find the enhanced trilinear estimate:
    \begin{equation*}\label{Equation: Trilinear HHL}
\|\prod_{j=1}^3e^{it\Delta}\phi_j\|_{L^2_{x,t}(\lambda\mathbb{T}\times\mathbb{T})}\lesssim \left(\frac{1}{\lambda} + \frac{M\vee |I_1|}{N_{\max}}\right)^\frac{1}{2}\prod_{j=1}^3\|\phi_j\|_{L^2_x(\lambda\mathbb{T})}.
    \end{equation*}
\end{proposition}
\begin{proof}
    This follows in a similar manner as in \cite{de2007global} \& \cite{li2011global}. Specifically, we begin by obtaining a bound on the size of $A$, defined as 
    \begin{multline*}
        A = \{n_1,n_2\in\mathbb{Z}/\lambda\,:\,n_1\in I_1,\,n_2\in I_2,\,n_3:=n-n_1-n_2\in I_3,\,|n_1-n_3|\gtrsim N_{13}\\
        |n_2-n_3|\gtrsim N_{23}, \tau-n_1^2-n_2^2-(n-n_1-n_2)^2 = O(1)\,\},
    \end{multline*}
    in the form of
    \[
    \sup_{n,\tau}\#A\lesssim \lambda^2\frac{K}{N_{\max}}+O(\lambda),
    \]
    for some $K\ll N_{\max}$. Once we have this in hand, we observe after taking the supremum and applying Cauchy-Schwarz that
    \[
\|\prod_{j=1}^3e^{it\Delta}\phi_j\|_{L^2_{x,t}(\lambda\mathbb{T}\times\mathbb{T})}\lesssim \left(\frac{1}{\lambda^2}\sup_{n,\tau}\#A\right)^\frac12\prod_{j=1}^3\|\phi_j\|_{L^2_x(\lambda\mathbb{T})},
    \]
    and hence we may conclude our estimate.

    Turning our attention to $\sup_{n,\tau}\#A$, we follow \cite{de2007global} in that we seek to transfer this question to a question about lattice points in a disk that also lie in an elliptical annulus. To that end, we fix an element $(n_1,n_2)\in A$, and its corresponding $n_3 = n-n_1-n_2$. We then will count the number of $\ell_1,\ell_2\in\mathbb{Z}$ such that $(n_1+\ell_1/\lambda, n_2+\ell_2/\lambda)\in A$. That is, 
    \begin{align*}
    n_1+\tfrac{\ell_1}{\lambda}&\in I_1,\\
    n_2+\tfrac{\ell_2}{\lambda}&\in I_2,\\
    n-n_1-n_2-\tfrac{\ell_1+\ell_2}{\lambda}&\in I_3,\\
    |n_1-n_3-\tfrac{2\ell_1+\ell_2}{\lambda}|&\gtrsim N_{13},\\
    |n_2-n_3-\frac{\ell_1+2\ell_2}{\lambda}|&\gtrsim N_{23}
    \end{align*}
    together with the property that 
    \[
    \tau - (n_1+\tfrac{\ell_1}{\lambda})^2-(n_2+\tfrac{\ell_2}{\lambda})^2 - (n_3-\tfrac{\ell_1+\ell_2}{\lambda})^2 = O(1).
    \]
    Expanding the above, we find that 
    \begin{multline*}
        (n_1+\tfrac{\ell_1}{\lambda})^2+(n_2+\tfrac{\ell_2}{\lambda})^2 + (n_3-\tfrac{\ell_1+\ell_2}{\lambda})^2\\
        = n_1^2+n_2^2+n_3^2+\frac{1}{\lambda^2}(\ell_1^2+\ell_2^2+(\ell_1+\ell_2)^2) + 2\frac{\ell_1}{\lambda}(n_1-n_3)+2\frac{\ell_2}{\lambda}(n_2-n_3)\\
        = \tau+O(1) + \frac{2}{\lambda^2}(\ell_1^2+\ell_2^2+\ell_1\ell_2)+2\frac{\ell_1}{\lambda}(n_1-n_3)+2\frac{\ell_2}{\lambda}(n_2-n_3),
    \end{multline*}
    and also that
    \begin{align*}
    \left|\frac{\ell_j}{\lambda}\right| &= \left|\frac{\ell_j}{\lambda}+n_j-n_j\right|\lesssim |I_j|,\quad j=1,2\\
    |n_1-n_3|&\gtrsim N_{13}\\
    |n_2-n_3|&\gtrsim N_{23}.
    \end{align*}
    Combining the above and relabeling $\tilde{a} = n_1-n_3$, $\tilde{b} = n_2-n_3$, we see that it is sufficient to count $x,y\in\mathbb{Z}$ satisfying:
    \begin{align}
        \big|x^2+y^2+xy+&\lambda(x\tilde{a}+y\tilde{b})\big|\leq c_1\lambda^2\label{Equation: Elliptical Annular Region Pre}\\
        |x| \leq c_2\lambda |I_1|&, \quad |y| \leq c_2\lambda|I_2|\label{Equation: Rectangular Region Pre}\\
        \tilde{a}^2+\tilde{b}^2&\sim N_{\max}^2,\nonumber
    \end{align}
    for some constants $c_1, c_2> 0$.
    
    One can view the first condition above as an elliptical annulus-- define $u = x-y$, $v = x+y$, $a = \tilde{a}-\tilde{b}$, and $b = \tfrac{1}{3}(\tilde{a}+\tilde{b})$, with which we may rewrite the portion within the absolute value as:
    \[
        x^2+y^2+xy+\lambda(x\tilde{a}+y\tilde{b}) = \frac{1}{4}\left(u+\lambda a \right)^2 + \frac{3}{4}\left(v + \lambda b\right)^2
        - \frac{\lambda^2}{4}a^2-\frac{3\lambda^2}{4}b^2.
    \]
    In particular, each such radius defines an ellipse centered at $(-\lambda a, -\lambda b)$ with major axis in the $x-y$ direction, minor axis in the $x+y$ direction, and extending $\sim \lambda N_{\max}$ in both directions. That is, it's morally a circle with radius $O(\lambda N_{\max})$. 

    We now examine the second condition above, which defines a rectangle centered at $(0,0)$. If $|I_1|\sim |I_2|$, then we can view this as nearly a circle with radius $O(\lambda |I_1|)$. However, if $|I_1|\ll |I_2|$, then this defines a long rectangle with major axis in the direction of the $y-$axis, and minor axis of length $|I_1|$. Gains in this region will come from the observation that, under certain conditions, we may shrink the intersection from the full rectangle to a much smaller disk. 
    

    We now demonstrate that if $|I_1|\ll |I_2|$ and condition \eqref{Condition: M condition}, as well as the conditions around it, hold, then the elliptical annulus defined via \eqref{Equation: Elliptical Annular Region Pre} intersects in a strictly smaller region than $B(0,|I_2|)$, and that the region this forms is contained in an $O(M\vee |I_1|)$ ball around the origin. To that end, let $c_1, c_2$ be the constants in equations \eqref{Equation: Elliptical Annular Region Pre} and \eqref{Equation: Rectangular Region Pre}, respectively, and assume without loss of generality that $\tilde{b} > 0$. We let $B(c_2) = c_2|I_1|+\tilde{b}$, and observe that substitution of $x =  c_2\lambda|I_1|$ into equation \eqref{Equation: Elliptical Annular Region Pre} yields:
    \[
    (y+\tfrac{\lambda}{2}B)^2 = \tfrac{\lambda^2B^2}{4}\left\{1+4\frac{c_1-c_2^2|I_1|^2-c_2|I_1|\tilde{a}}{B^2}\right\}.
    \]
    Where we realize the inner term satisfies
    \[
    \left|\frac{c_1-c_2^2|I_1|^2-c_2|I_1|\tilde{a}}{B^2}\right|\lesssim \frac{M}{N_{23}}\ll 1,
    \]
    and hence we may unwind and utilize a Taylor series bound to conclude
    \[
y(c_1, c_2) = -\frac{\lambda}{2}\left\{c_2|I_1|+\tilde{b} - |c_2|I_1|+\tilde{b}|\right\} + O\left(\lambda M\right)
 = O\left(\lambda M\right).
    \]
    A similar argument allows us to conclude that $y(j_1c_1, j_2 c_2) = O(\lambda M)\ll \lambda |I_2|$ for any $j_1,j_2\in\{-1,1\}$. 
    
    This observation implies that the intersection is contained in the region $|x|, |y| = O(\lambda(|I_1|\vee M))$, and hence will live in a smaller ball with equation given by:
    \[
    x^2+y^2\leq c_3\lambda^2(|I_1|\vee M)^2,
    \]
    where we are free to write this in terms of our new variables, up to some multiplicative loss on $c_3$.
    
    To summarize, we have reduced our original system to finding the number of $u,v\in\mathbb{Z}$ such that:
    \begin{align}
        (u+\lambda a)^2+3(v+\lambda b)^2&= R +\lambda^2\left(a^2 + 3b^2\right)\quad |R|\leq c\lambda^2\label{Equation: Ellipse Shell}\\
        u^2+v^2&\leq c_3\lambda^2 K^2\label{Equation: Circle}\\
        a^2+b^2&\sim N_{\max}^2\label{Equation: a b condition},
    \end{align}
    where we take
    \begin{equation}\label{Equation: M Definition}
K := \begin{cases}
    |I_1|\vee M & \mbox{if }\eqref{Condition: M condition} \mbox{ holds}\\
    |I_2| & \mbox{else}.
\end{cases}
    \end{equation}
    We remark that, pictorially and philosophically, this is only a minor variation of \cite[Figures 2 \& 3]{de2007global}.

    We now seek to count the number of elements in the region formed by conditions \eqref{Equation: Ellipse Shell}-\eqref{Equation: a b condition}. For fixed $a,b$, we define the set 
    \[
    A^\lambda := \{(u,v)\in\mathbb{Z}^2\,:\eqref{Equation: Ellipse Shell}-\eqref{Equation: a b condition}\,\mathrm{holds}\},
    \]
    and observe that $A^\lambda = \lambda A^1$, where
    \begin{multline*}
    A^1 = \{(u,v)\in\mathbb{Z}^2\,:\,(u+a)^2+3(v+b)^2= R +\left(a^2 + 3b^2\right),\,|R|\leq c,\\
    \,u^2+v^2\leq c_3K^2,\,a^2+b^2\sim N_{\max}^2\}.
    \end{multline*}
    Hence, the desired claim would follow if we showed
    \begin{equation}\label{Equation: Desired Bound}
|A^\lambda|\lesssim \lambda^2\frac{K}{N_{\max}} + O(\lambda),
    \end{equation}
    where $K$ is as in \eqref{Equation: M Definition}. To that end, we let $\theta$ denote the angle between the two tangent lines to the circle
    \[
    u^2+v^2= c_3K^2
    \]
    which pass through the point $(-a,-b)$, and note that
    \[
    \sin\theta \lesssim \frac{K}{N_{\max}}.
    \]
    It follows that for $N_{\max}\gg K$ we have $\theta \lesssim K/N_{\max}\ll 1$. 
    
    Let $S(R)$ for $|R| \leq c$ denote the the sector of the ellipse 
    \[
    (u+a)^2+3(v+b)^2= R +\left(a^2 + 3b^2\right)
    \]
    whose boundaries are formed by the two tangent lines above and the boundary of the ellipse. In this way, $S(R)$ is a sector of the ellipse with angle given by $\theta$, and that $A^1\subset\mathbb{Z}^2\cap S(c)\setminus S(-c)$.

    Combining the above together with a theorem of Gauss (see, e.g. \cite{MR1945283}), we find that
    \[
    |A^\lambda|= |\mathbb{Z}^2\cap \lambda A^1|\leq |\mathbb{Z}^2\cap \lambda S(c)| - |\mathbb{Z}^2\cap\lambda S(-c)| = \lambda^2 |S(c)\setminus S(-c)| + O(\lambda).
    \]
    However, as the axis are similar in size we have (from, say, a trivial integral calculation) that 
    \[
    |S(c)\setminus S(-c)|\lesssim \theta,
    \]
    and hence
    \[
|A^\lambda|\lesssim \lambda^2\frac{K}{N_{\max}}+O(\lambda),
    \]
    which is exactly the desired bound, \eqref{Equation: Desired Bound}.
\end{proof}
\begin{remark}
    It's worth noting that the above does not require all intervals present to be disjoint, but does require some amount of seperation. For example, the first conclusion holds even when $I_2\subset I_3$, so long as $\operatorname{dist}(I_3, I_1)\gg |I_2|$. The key is simply that we have enough seperation to guarantee that $|I_2|/N_{\max}$ (resp. $M\vee |I_1|/N_{\max}$) is small.
\end{remark}

The extension of the trilinear estimates of Proposition \ref{Proposition: Trilinear Free Estimate} to the $Y^s$ spaces follows by a standard scaling argument, together with the transference principle and interpolation result of Proposition \ref{Proposition: Interpolation Properties}.
\begin{proposition}\label{Lemma: trilinear Estimates}
    Suppose $s < \frac{\alpha}{2}$, and let $L\in 2^{\mathbb{N}_0}$, $I_1, I_2, I_3$ be intervals contained in $[-10L, 10L]$ and $M$ satisfy the conditions of Proposition \ref{Proposition: Trilinear Free Estimate}, with
    \[
    K := 
\begin{cases}
    |I_1|\vee M & \mbox{if }\eqref{Condition: M condition}\\
    |I_2| &\mbox{else}.
\end{cases}
    \]
    Then, for
    \[
    N_{13} = \mathrm{dist}(I_1, I_3)\quad\mbox{ and }\quad N_{23} = \mathrm{dist}(I_2, I_3)
    \]
    with $K\ll \max(N_{13}, N_{23})$, we have the following trilinear Strichartz estimate:
    \[
    \|P_{I_1}u_1P_{I_2}u_2P_{I_3}u_3\|_{L^2_{x,t}(\mu\mathbb{T}\times\lambda\mathbb{T})}\lesssim \log^3(\langle \lambda L\rangle)\left(\frac{1}{N}+\frac{K}{\max(N_{13}, N_{23})}\right)^\frac{1}{2}\prod_{j=1}^3\|P_{I_j}u_j\|_{Y^0_{+,\mu}}.
    \]
    Moreover, equivalent statements hold for conjugates so long as one has the analogous statement for $-I_j$.
\end{proposition}
\begin{proof}
    The cornerstone of this analysis will be the estimates of Proposition \ref{Proposition: Trilinear Free Estimate}.

    Observe that it suffices to prove:
    \begin{align*}
        \|P_{I_1}u_1P_{I_2}u_2P_{I_3}\|_{L^2_{t,x}(\mu\mathbb{T}\times\lambda\mathbb{T})}&\lesssim \lambda^\beta \langle \lambda M\rangle^{0+}\prod_{j=1}^3\|P_{I_j}u_j\|_{U^p_\Delta}\\
        \|P_{I_1}e^{it\Delta}\phi_1P_{I_2}e^{it\Delta}\phi_2P_{I_3}e^{it\Delta}\phi_3\|_{L^2_{t,x}(\mu\mathbb{T}\times\lambda\mathbb{T})}&\lesssim \left(\frac{1}{N}+\frac{K}{\max(N_{13}, N_{23})}\right)^\frac{1}{2}\prod_{j=1}^3\|P_{I_j}\phi_j\|_{L^2_x(\lambda\mathbb{T})}
    \end{align*}
    for some $\beta > 0$ and $p > 2$ and all $\phi_1, \phi_2\in L^2(\lambda\mathbb{T})$. Indeed, the second bound above and the first portion of Proposition \ref{Proposition: Interpolation Properties} yields boundedness of the operator from $U^2_\Delta\times U^2_\Delta\times U^2_\Delta\to L^2$, while interpolation by the second portion of Proposition \ref{Proposition: Interpolation Properties} yields the desired claim.

    The first follows from the standard $L^6_{t,x}(\mathbb{T}\times\lambda\mathbb{T})$ with loss of Lemma \ref{Proposition: Scaled Strichartz}. 
    The second portion follows from a scaling argument and the estimate of Proposition \ref{Proposition: Trilinear Free Estimate}. Indeed, we observe that
    \begin{multline*} \|P_{I_1}e^{it\Delta}\phi_1P_{I_2}e^{it\Delta}\phi_2P_{I_3}e^{it\Delta}\phi_3\|_{L^2_{t,x}(\mu\mathbb{T}\times\lambda\mathbb{T})}\\
    = \|P_{\sqrt{\mu}I_1}e^{it\Delta}\phi_1^{\sqrt{\mu}^{-1}}P_{\sqrt{\mu}I_2}e^{it\Delta}\phi_2^{\sqrt{\mu}^{-1}}P_{\sqrt{\mu}I_3}e^{it\Delta}\phi_3^{\sqrt{\mu}^{-1}}\|_{L^2_{t,x}(\mathbb{T}\times\tfrac{\lambda}{\sqrt{\mu}}\mathbb{T})}\\
    \lesssim \left(\frac{\sqrt{\mu}}{\lambda} + \frac{K}{\max(N_{13}, N_{23})}\right)^\frac12\prod_{j=1}^3\|P_{\sqrt{\mu}I_j}\phi_j^{\sqrt{\mu}^{-1}}\|_{L^2_x(\tfrac{\lambda}{\sqrt{\mu}}\mathbb{T})}\\
    \lesssim \left(\frac{1}{(\lambda N)^\frac{\alpha}{2}}+\frac{K}{\max(N_{13}, N_{23})}\right)^\frac{1}{2}\prod_{j=1}^3\|P_{I_j}\phi_j\|_{L^2_x(\lambda\mathbb{T})},
    \end{multline*}
    where the penultimate step is an application of Lemma \ref{Proposition: Trilinear Free Estimate} after the observation that $K\mapsto \sqrt{\mu}K$ and $\max(N_{13}, N_{23})\mapsto \sqrt{\mu}\max(N_{13}, N_{23})$. We observe that
    \[
    (\lambda N)^\frac{\alpha}{2} = N^\frac{\alpha}{2s}\gtrsim N\qquad\qquad\frac{\alpha}{2} > s,
    \]
    and hence we recover the claimed result.
\end{proof}
\begin{remark}
    Proposition \ref{Lemma: trilinear Estimates} has a factor of $\lambda$ in it. However, we will take $\lambda \sim N^{\frac{1-s}{s}+}\gtrsim N$ for $s < \frac12$, so it suffices to replace $\lambda^{-1}$ with $N^{-1}$. Moreover, as $\log\lambda\sim \log N$ for fixed $s$, we will always be able to deal with this as $L^{0+}$ for $L\gtrsim N$.
\end{remark}
\begin{remark}
    It's possible that a more refined version of the number of lattice points in a scaled compact convex body could improve the $O(\tfrac{1}{\lambda})$ error bound that imposed the $s < \tfrac{\alpha}{2}$ constraint in the above. We are able to take $s > \tfrac{\alpha}{3}$, and hence there would be no gain in pursuing this here.
\end{remark}
The main use of Proposition \ref{Lemma: trilinear Estimates} will be in the $HLL$ and a subset of the $HHL$ (when have that both $N_{13}\sim N_{23}\sim N_{\max}\gg |I_1|$) cases. In these cases we can always obtain the result of the enhanced version, as when $|I_1|\sim |I_2|$ the base estimate holds, and when $|I_1|\ll |I_2|$ we have the conditions of Lemma \ref{Lemma: trilinear Estimates} with
\[
M\lesssim \frac{|I_1|(N_{\max}+|I_1|)}{N_{\max}}\lesssim |I_1|\ll |I_2|.
\]
On the other hand, there will be exactly $1$ case, $A_{3214b}$ of Lemma \ref{Lemma: M6 integral term}, in which we may not apply the enhanced version as above. In this case, we have that $N_{23}$ isn't quite large enough to fully yield $|I_1|/N_{\max}$. These observations are recorded in the next corollary.

\begin{corollary}\label{Corollary: Trilinear Strichartz Specific}
    Suppose that $s < \tfrac{\alpha}{2}$, $L\in 2^{\mathbb{N}_0}$, and $I_j\subset [-10L, 10L]$ for $j = 1, 2,3$ are intervals satisfying $|I_1|\lesssim |I_2|\lesssim |I_3|$. Suppose further that
    \[
    \mathrm{dist}(I_1, I_3)\sim \mathrm{dist}(I_2, I_3)\sim L,
    \]
    and that $|I_1|\ll L$. Then we have the following trilinear Strichartz estimate:
    \[
    \|P_{I_1}u_1P_{I_2}u_2P_{I_3}u_3\|_{L^2_{x,t}(\mu\mathbb{T}\times\lambda\mathbb{T})}\lesssim \log^3(\langle \lambda L\rangle)\left(\frac{1}{N}+\frac{|I_1|}{L}\right)^\frac{1}{2}\prod_{j=1}^3\|P_{I_j}u_j\|_{Y^0_{+,\mu}},
    \]
    with analogous results for conjugates so long as there are analogous statements with $-I_j$.
\end{corollary}
\begin{proof}
    This follows by the above. In particular, if $|I_1|\sim |I_2|$ then we may apply the base estimate. Suppose now that $|I_1|\ll |I_2|$ and $\mathrm{dist}(I_1, I_3)\sim \mathrm{dist}(I_2, I_3)\sim L$, then $I_1-I_3\subset B(0, 20L)$ and
    \[
M \lesssim \frac{|I_1|(L+|I_1|)}{L}\lesssim |I_1| \ll |I_2|.
    \]
    An application of the enhanced version of the trilinear estimate yields the desired result.
\end{proof}

\section{Modified Energy, Pointwise Estimates, \& Proof of Proposition \ref{Proposition: I method bounds}}\label{Section: I-method Setup}
With the prior lemmas out of the way, we turn to the construction of the modified energy described in Proposition \ref{Proposition: I method bounds}. Due to conservation of energy for \eqref{Equation: Quintic NLS}, we may restrict all of the multilinear operators to:
\[
\Upsilon_j = \{(k_1,\dots,k_j)\in\Gamma_j\,:\,|k_1^*|\sim|k_2^*|\gtrsim N\},
\]
where we necessarily have cancellation in the low frequency component.

They key theme in what is to follow is that we seek to perform integration-by-parts to exploit large phases. However, the phases under consideration are of of the form:
\begin{equation}\label{Definition: alpha}
\Omega_n = i\sum_{j=1}^n(-1)^jk_j^2,
\end{equation}
which have non-trivial \textit{resonances}. Even for just $n = 6$ and $k_j\in\mathbb{Z}$, the set where $\alpha_n = 0$ or is simply small is hard to accurately describe; which is made even worse on $\tfrac{1}{\lambda}\mathbb{Z}$. At the end of the construction we will need to determine a pointwise bound for the quotient
\begin{equation}\label{Equation: Quotient to estimate}
\frac{m^2(k_1)k_1^2-m^2(k_2)k_2^2+m^2(k_3)k_3^2 - m^2(k_4)k_4^2+m^2(k_5)k_5^2-m^2(k_6)k_6^2}{k_1^2-k_2^2+k_3^2-k_4^2+k_5^2-k_6^2},
\end{equation}
where we will not only desire this to be relatively small, but we will desire finer bounds for this quantity than in \cite{li2011global,schippa2024improved}; that is, small gains alleviate the necessity of $\frac13 > s > 0$. 

To overcome this issue, we define the collection of unacceptable frequency configurations as in \cite{schippa2024improved}. It is worth noting that this resonance set is nearly the same as in \cite{li2011global}, merely with the language modified.
\begin{definition}\label{Definition: Resonant Set}
    Let $N_j\in2^{\mathbb{N}_0}$ and $(k_1,\dots,k_6)\in\Upsilon_6$ satisfying $|k_j|\sim N_j$. Moreover, assume that $N_1\geq N_3\geq N_5$ and $N_2\geq N_4\geq N_6$, and that
    \[
    N_1\sim N_2,\quad\quad N_3^*\sim N_4^*.
    \]
    We then define the Resonance, $\mathcal{R}$, as one of the following holding:
    \begin{enumerate}[label = (\roman*)]
        \item $N_1^*\sim N_2^*\gg N_3^*\sim N_4^*$, $k_1\cdot k_2 < 0$, and $|k_1+k_2|\ll\frac{N_3^*}{N_1^2}$,
        \item $N_1^*\sim N_4^*\gg N_5^*$ and one of the following is true:
        \begin{enumerate}
            \item $\{N_1^*,\dots,N_4^*\} = \{N_1,N_2,N_3,N_4\}$.
            \item $\{N_1^*,\dots,N_4^*\} = \{N_1,N_2,N_4,N_6\}$, and $k_2,k_4,k_6$ are not of the same sign. Moreover, if $k_i$ and $k_1$ for $i\in\{2,4,6\}$ are of the same sign, then $|k_1-k_i|\sim N_1^*$. If $k_i$ and $k_1$ differ in sign, then $|k_1+k_i|\sim N_1^*$.
            \item $\{N_1^*,\dots,N_4^*\} = \{N_1,N_2,N_3,N_5\}$, and $k_1,k_3,k_5$ are not of the same sign. Moreover, if $k_i$ and $k_2$ for $i\in\{2,4,6\}$ are of the same sign, then $|k_2-k_i|\sim N_1^*$. If $k_i$ and $k_2$ differ in sign, then $|k_2+k_i|\sim N_1^*$.
        \end{enumerate}
        \item $N_1^*\sim N_5^*$.
    \end{enumerate}
\end{definition}
\begin{remark}
    When off of $\mathcal{R}$, we can estimate \eqref{Equation: Quotient to estimate} by utilizing the large denominator, which is the content of Lemma \ref{Lemma: spatial symbol size estimate}.
\end{remark}

With the definition of resonance out of the way, we can now proceed with the construction of the modified energies. For any symbol $M_n$, we recall the definition of $\Lambda_n$ defined by \cite{de2007global} as:
\[
\Lambda_n(M_n\,:\,f_1,\dots, f_n) = \int_{\Gamma_n} M_n(k_1,\dots, k_n)\prod_{j=1}^n\widehat{f}_j(t,k_j)(d\Gamma_n)_\lambda.
\]
If $u$ satisfies \eqref{Equation: Quintic NLS}, then we find
\[
\partial_t\Lambda_n(M_n;u,\dots, u) = \Lambda_n(M_n\Omega_n) + i\Lambda_{n+4}\left(\sum_{j=1}^n(-1)^jX_j(M_n);u,\dots,u\right),
\]
for $\Omega_n$ as in \eqref{Definition: alpha} and $X_j(M_n)$ the \textit{elongation} of $M_n$ in the $j'$th coordinate:
\[
X_j(M_n) = M_n(k_1,\dots, k_{j-1},k_j+\dots +k_{j+5},k_{j+6},\dots, k_{n+4}).
\]
We then define the first energy to be 
\[
E^1_I(u)(t) = E(Iu)(t) = \frac{1}{2}\|Iu(t)\|_{\dot{H}^1}^2\pm\frac{1}{6}\|Iu(t)\|_{L^6}^6 = \Lambda_2(\sigma_2)+ \Lambda_6(\sigma_6),
\]
where
\begin{align*}
    \sigma_2 &= -\frac{1}{2}m(k_1)k_1m(k_2)k_2,\quad\mbox{and}\\
    \sigma_6 &= \pm\frac{1}{6}m(k_1)m(k_2)m(k_3)m(k_4)m(k_5)m(k_6).
\end{align*}
We now define $M_6$, via
\[
M_6 =  \frac{i}{6}\sum_{j=1}^6(-1)^{j+1}m^2(k_j)k_j^2 + \Lambda_6(\sigma_6\Omega_6) =: M^1_6 + M^2_6,
\]
and observe that $M_6 = M_6\chi_{\Upsilon_6}$, so that we are free to assume that all of the multipliers come with a restriction to $\Upsilon_6$. We now introduce the resonant set into the mix, by defining 
\[
\overline{M}_6 := \sum_{\substack{I_1,\dots, I_6\\I_i\subset A_{N_i}\\(\mathcal{R})\,\mathrm{ holds}}}M^1_6(k_1,\dots,k_6)\chi_{I_1}(k_1)\dots\chi_{I_6}(k_6),
\]
where $|I_i| \sim N_i$, unless a finer decomposition can be admitted in accordance with Lemma \ref{Lemma: M6 estimate}. This then induces the portion of $M_6$ that we can remove via:
\[
\tilde{M}_6 := M_6 - \overline{M}_6,
\]
which contains all of $M_6^2 = \Lambda_6(\sigma_6\Omega_6)$ as well the \textit{nonresonant} portion of $M_6^1$. The resonant set was constructed so that we may sensibly bound and define $\tilde{\sigma}_6$ and the second modified energy as:
\begin{align*}
    \tilde{\sigma_6} &:= \frac{\tilde{M}_6}{\Omega_6}\\
    E^2_I(u)(t) &:= E^1_I(u)(t) - \Lambda_6(\tilde{\sigma}_6).
\end{align*}

The use in these definitions arises after differentiation in time, where we find major cancellation within the above:
\begin{align*}
\frac{d}{dt}E^2_I(u)(t) &= \frac{d}{dt}E^1_I(u)(t) - \frac{d}{dt}\Lambda_6(\tilde{\sigma}_6)\\
&=\Lambda_6\left(\frac{i}{6}\sum_{j=1}^6(-1)^{j+1}m^2(k_j)k_j^2\right) + \Lambda_6(\sigma_6\Omega_6) + i\Lambda_{10}\left(\sum_{j=1}^6(-1)^{j}X_j(\sigma_6)\right)\\
&\qquad\qquad - \Lambda_6(\tilde{M}_6) - i\Lambda_{10}\left(\sum_{j=1}^6(-1)^jX_j(\tilde{\sigma}_6)\right)\\
&= \Lambda_6(\overline{M}_6)+ i\Lambda_{10}\left(\sum_{j=1}^6(-1)^j(X_j(\sigma_6)-X_j(\tilde{\sigma}_6))\right).
\end{align*}
We may now apply the Fundamental Theorem of Calculus, which allows us to obtain a bound on the increment of the energy in terms of the above multilinear operators:
\begin{multline}
    E^1_I(u)(t) = E^1_I(u)(0) + \Lambda_6(\tilde{\sigma}_6)\big|_{s=0}^t\\
    + \int_0^t \Lambda_6(\overline{M}_6)(s)
    + i\Lambda_{10}\left(\sum_{j=1}^6(-1)^j(X_j(\sigma_6)-X_j(\tilde{\sigma}_6))\right)(s)\,ds,
\end{multline}
where we designate 
\[
M_{10} = \sum_{j=1}^6(-1)^j(X_j(\sigma_6)-X_j(\tilde{\sigma}_6)).
\]
This is the exact expression given by equation \eqref{Equation: I method FTC}, so we are now in a position to prove the individual bounds associated to Proposition \ref{Proposition: I method bounds}.

\subsection{Pointwise Symbol Estimates}\label{Section: Pointwise Symbol Estimates}
In this section we prove pointwise symbol estimates associated to the multilinear operators within equation \eqref{Equation: I method FTC}. For convenience, we redefine 
\[
\Omega_n(k_1,\dots, k_n)\mapsto i\Omega_n(k_1,\dots, k_n) = \sum_{j=1}^n(-1)^{j+1}k_j,
\]
as there will be no confusion.

We now record a proposition that will enable us to easily use regularity and pointwise bounds to estimate multilinear operators in our spaces. This proposition is standard, and is the content of \cite[Proposition 5.4]{schippa2024improved}.
\begin{proposition}\label{Proposition: Regularity}
    Let $I_1,\dots, I_6\subset \mathbb{R}$ denote intervals such that $|I_j| = L_j\in 2^\mathbb{Z}$. Let $M:\mathbb{Z}^6\to\mathbb{C}$, 
    \[
    M_I(k_1,\cdots, k_6) = M(k_1,\dots, k_6)\prod_{j=1}^6\chi_{I_j}(k_j),
    \]
    and suppose that there is an extension of $M_I$, $\overline{M}_I$, on $\mathbb{R}^6$ and $\cal{M}(I_1,\dots, I_6)>0$ with:
    \begin{equation}\label{Equation: Regularity}
        |\partial_{\ell_j}^\alpha \overline{M}_I(\ell_1,\dots,\ell_6)|\lesssim \frac{\cal{M}(I_1,\dots,I_6)}{L^\alpha_j}
    \end{equation}
    for $0 \leq \alpha \leq 15$ and $1\leq j\leq 6$.

    Then we have the Fourier series expansion:
    \begin{equation}
        \overline{M}_I(k_1,\dots, k_6) = \frac{1}{L_1\cdots L_6}\sum_{\xi_1,\dots, \xi_6\in\mathbb{Z}}m(\xi_1,\dots,\xi_6)\prod_{j=1}^6e^{i\frac{k_j\xi_j}{L_j}}
    \end{equation}
    whose Fourier coefficients are given by
    \begin{equation}
        m(\tfrac{\xi_1}{L_1},\dots,\tfrac{\xi_6}{L_6}) = \int_{\ell_j\in [0,L_j]}\overline{M}(\ell_1,\dots,\ell_6)\prod_{j=1}^6e^{-i\frac{\xi_j\ell_j}{L_j}}d\ell_1\dots d\ell_6.
    \end{equation}
    Moreover, 
    \begin{equation}
        |m(\tfrac{\xi_1}{L_1},\cdots,\tfrac{\xi_6}{L_6})|\lesssim L_1\dots L_6 \cal{M}(I_1,\dots, I_6)\langle \xi\rangle^{-15}.
    \end{equation}
\end{proposition}
The above proposition, combined with the translation invariance of Strichartz estimates, allows us to estimate multilinear operators with symbol $M$ via dyadic decomposition and the pointwise estimates of $M$, see \cite[Remark 5.5]{schippa2024improved}.

The next lemma provides the base estimates for the resonant operator $\overline{M}_6$, which is only an improvement in the addition of certain factors of $m$ in cases $(iii)$ and $(iv)$.
\begin{lemma}\label{Lemma: M6 estimate}
    Let $I_1\dots, I_6$ denote intervals of length $L_i\in 2^\mathbb{Z}$. The multiplier 
    \[
    \overline{M}_6(k_1,\dots, k_6)\chi_{I_1}(k_i)\cdots\chi_{I_6}(k_6)
    \]
    defined on $\Gamma_6$ extends to $\mathbb{R}^6$ and satisfies size and regularity assumptions under the following conditions on $I_1,\dots I_6$.
    \begin{enumerate}[label = (\roman*)]
        \item For $|k_i|\sim N_i\in 2^{\mathbb{N}_0}$ and $N_1^*\sim N_2^*$ we have the size estimate and regularity estimates \eqref{Equation: Regularity} with 
        \begin{equation}\label{equation: M bound 1}
            |\cal{M}(I_1,\dots,I_6)|\lesssim m(N_1^*)N_1^*m(N_3^*)N_3^*.
        \end{equation}
        The same holds if $N_1^*\gg N_3^*$ and $|I_1|\sim |I_2|\sim N_3^*$.
        \item If $N_1^*\sim |k_1|\sim |k_2|\gtrsim N\gg N_3^*\sim N_4^*$, $|I_1|\sim |I_2|\lesssim \tfrac{(N_3^*)^2}{N_1^*}$, and $|k_1+k_2|\lesssim \tfrac{(N_3^*)^2}{N_1^*}$ for $k_j\in I_j$, then we have the size and regularity estimate with
        \begin{equation}\label{Equation: M bound 2}
            |\cal{M}(I_1,\dots, I_6)|\lesssim (N_3^*)^2.
        \end{equation}
        \item If $|I_j|\lesssim N_5^*$ for $j=1,2,3,4$, and $\max(|k_1+k_2|,|k_3+k_4|)\lesssim N_5^*$, then we have the size and regularity estimate with
        \begin{equation}\label{Equation: M bound 3}
            |\cal{M}(I_1,\dots, I_6)|\lesssim m(N_1^*)N_1^*m(N_5^*)N_5^*.
        \end{equation}
        \item If $|I_j|\lesssim N_{12}$ for $j = 1,2,3,4$, and $N_1\gtrsim N_{12}\sim |k_1+k_2|\sim |k_3+k_4|\gg |k_5^*|$, then the size and regularity estimate hold with
        \begin{equation}\label{Equation: M bound 4}
            |\cal{M}(I_1,\dots, I_6)\lesssim m(N_1^*)N_1^*m(N_{12})N_{12}.
        \end{equation}
    \end{enumerate}
\end{lemma}
\begin{proof}
    The statements, and hence the proofs, for $(i)$ and $(ii)$ are the same as in \cite{schippa2024improved}. What remains is $(iii)$ and $(iv)$.

    Before proceeding with the cases, we note the basic estimate that will be utilized:
    \begin{align}
        &|m^2(k_1)k_1^2-m^2(k_2)k_2^2 + \cdots + m^2(k_5)k_5^2-m^2(k_6)k_6^2|\nonumber\\
        &\qquad\lesssim |m^2(k_1)k_1^2 - m^2(k_2)k_2^2| + |m^2(k_3)k_3^2 - m^2(k_4)k_4^2| + |m^2(k_5)k_5^2-m^2(k_6)k_6^2|\nonumber\\
        &\qquad\lesssim m^2(N_1^*)|k_1^2-k_2^2| + m^2(N_3^*)|k_3^2-k_4^2| + m^2(N_5^*)|k_5^2-k_6^2|\label{Equation: symbol start}
    \end{align}

    \textbf{Case (iii):} The proof for this is almost the exact same, but we modify it to take advantage of the extra factor of $m$. We have that $|k_1+k_2|,|k_3+k_4|\lesssim N_5^*$, so
    \[
        \eqref{Equation: symbol start}\lesssim m^2(N_1^*)N_1^*N_5^* + m^2(N_3^*)N_3^*N_5^* + m^2(N_5^*)(N_5^*)^2\\
        \lesssim m(N_1^*)N_1^*m(N_5^*)N_5^*,
    \]
    where we've used both the decreasing nature of $m$ and the increasing nature of $m(|k|)|k|$. This establishes the bound, but we now need to show that it satisfies the regularity property of \eqref{Equation: Regularity}. To that end, we observe that we have the standard derivative estimate for the obvious extension of $\overline{M}_6$
    \[
    |\partial^\alpha_{x_i} \overline{M}_6(x_1, \dots x_6)|\lesssim \frac{m(x_i)^2x_i^2}{\langle x_i\rangle^\alpha},
    \]
    so that if $|x_i|\sim N_i\ll N_5^*$ then we may use the increasing nature of $m(|x|)|x|$ to conclude the desired result. On the other hand, if $|x_i|\sim N_i \gtrsim N_5^*$ then we're localizing to intervals of size $N_5^*$ and it suffices to show that
    \[
    \frac{m^2(x_i)x_i^2(N_5^*)^\alpha}{\langle x_i\rangle^{\alpha}}\lesssim 1,
    \]
    but we may again use the increasing nature of $m(|x|)|x|$ together with the decreasing nature of $m(|x|)$ to reduce us to having to show
    \[
    \frac{|x_i|(N_5^*)^{\alpha-1}}{\langle x_i\rangle^{\alpha}}\lesssim 1.
    \]
    This, however, is immediate from the fact that $\langle x_i\rangle \gtrsim N_5^*$.

    \textbf{Case (iv):} This follows from the same work as above after noting that we may not have $N_3^*\ll N_{12}$, and hence $m(N_3^*)\lesssim m(N_{12})$.
\end{proof}
\begin{remark}\label{Remark: N12 M term}
    The factor of $m(N_{12})$ in \eqref{Equation: M bound 4} is necessary to extending Theorem \ref{Theorem: GWP theorem} to $s < \tfrac{1}{3}$. In fact, this allows us to conclude that
    \[
    \frac{N^s}{N_1^s}\sum_{N_{12}\lesssim N_1}\frac{m(N_{12})N_{12}}{N} \lesssim 1 + \frac{N^s}{N_1^s}\sum_{N\ll N_{12}\lesssim N_1}N^{-s}N_{12}^s\lesssim 1,
    \]
    for some $N_1\gtrsim N$. This means we may utilize another factor of $N_1^{-s}$ to cancel this factor, and the above factor of $N^{-s}$ to cancel the $N^s$ factor that is now without a partner. 
\end{remark}
The next lemma is elementary and allows us to handle the regularity of terms involving $\Omega_p$ for some $p$.
\begin{lemma}\label{Lemma: Omega derivative}
    Let $p  > 3$ even,  $\Omega_p(\xi_1,\dots,\xi_p) = \sum_{j=1}^p (-1)^{j+1}\xi_j^2$, and $\tilde{\Omega}_p$ any polynomial extension of $\Omega_p$ to $\mathbb{R}^p$. Then for any $\alpha \geq 0$ and $1\leq i\leq p$ we have \begin{equation}
        \partial_{\xi_i}^\alpha\tilde{\Omega}_p^{-1} = \sum_{j}c_j\frac{(\partial_{\xi_i}\tilde{\Omega}_p)^{\gamma_j}}{\tilde{\Omega}_p^{\beta_j}},
    \end{equation}
    where $2\beta_j - \gamma_j \geq \alpha + 2$ and $\beta_j \geq \gamma_j + 1$. Furthermore, if $\xi_i\in I_i$ with $|I_i|\sim L_i$ and $|\xi_i|\sim N_i$, then we have the estimate $|\partial^\alpha_{\xi_i} \tilde{\Omega}^{-1}_p|\lesssim |\tilde{\Omega}_p|^{-1}L_i^{-\alpha}$ so long as:
    \begin{equation}\label{Equation: Regularity Omega desired est}
        |\partial_{\xi_i}\tilde{\Omega}_p|^{\gamma_j}L^{\alpha}_i\lesssim |\tilde{\Omega}_p|^{\beta_j-1},
    \end{equation}
    for all $\gamma_j,\beta_j$ as above.
\end{lemma}
\begin{proof}
    This is a trivial proof by induction. The base case is immediate, so we now proceed for some $\alpha > 0$, where we have
    \[
    \partial_{\xi_i}^\alpha \tilde{\Omega}_p^{-1} = \sum_j c^1_j \frac{(\partial_{\xi_i}\tilde{\Omega}_p)^{\gamma_j-1}\partial_{\xi_i}^2\tilde{\Omega}}{\tilde{\Omega}_p^{\beta_j}} + \sum_j c^2_j\frac{(\partial_{\xi_i}\tilde{\Omega}_p)^{\gamma_j+1}}{\tilde{\Omega}_p^{\beta_j+1}}
    \]
    with $2\beta_j-\gamma_j\geq \alpha+1$ and $\beta_j\geq \gamma_j + 1$ by the inductive hypothesis at $\alpha - 1$. Since $\partial_{\xi}^2\tilde{\Omega}_p = c(\tilde{\Omega})$ and both hypothesis on $\gamma_j$ and $\beta_j$ hold trivially, we conclude the proof.
\end{proof}

We now record an improvement to \cite[Lemma 5.8]{schippa2024improved}, which is similar to Lemma \ref{Lemma: M6 estimate}. In particular, this lemma provides a pointwise estimate for the nonresonant symbol $\tilde{\sigma}_6$. 
\begin{lemma}\label{Lemma: spatial symbol size estimate}
    The non-resonant operator $\tilde{\sigma_6}$ satisfies the size and regularity estimate:
    \begin{equation}
    \left|\tilde{\sigma}_6(k_1, \dots, k_6)\right|\lesssim 
    \begin{cases}
        m^2(N_3^*) & N_1\sim N_2\gg N_3^*\sim N_4^*\\
        m(N_1^*)m(N_3^*) & \mbox{else}.
    \end{cases}
    \end{equation}
\end{lemma}
\begin{proof}
    The tools for this were observed, but unused in \cite{schippa2024improved}. Before proceeding with the case work, we observe several facts which will enable this to go smoothly. The first is that the above lemma is immediate for the $\sigma_6$ portion of $\tilde{\sigma}_6$. Secondly, we note that for any $M \gtrsim L_i$:
    \[
    |\partial_{\xi_i}^{\alpha}M_6^1|\lesssim \frac{m^2(N_i)N_i^2}{N_i^{\alpha}}\lesssim \frac{m(N_1^*)N_1^*m(M)M}{N_i^{\alpha-1}L_i}\quad \alpha > 0,
    \]
    and also that in every case $|\Omega_6|\gtrsim N_1^*M$ (where these two will always coincide). Thus, it is sufficient to demonstrate both the estimate
    \[
    \left|\frac{M_6^1}{\Omega_6}\right|\lesssim m(N_3^*)m(M)\quad 1\leq i\leq 6
    \]
    for an appropriate $M\gtrsim L_i$, together with an extension, $\tilde{\Omega}_6$, of $\Omega_6$ that satisfies the regularity bound \eqref{Equation: Regularity Omega desired est}.\\
    
    \textbf{Case $N_1\gg N_2$:} In this case we find that trivially have the size estimate 
    \[
    |M^1_6/\Omega_6|\lesssim m^2(N_1).
    \]
    Regularity then follows from the regularity of both $\Omega_6^{-1}$ and $M^1_6$. Indeed, this follows from an application of Lemma \ref{Lemma: Omega derivative} and the observation that \eqref{Equation: Regularity Omega desired est} holds with $\partial_{\xi_i}\tilde{\Omega}_6 = 2\xi_i$, $|\xi_i|\sim N_i$, and $L_i\sim N_i$ after noting that $2\beta_j-\gamma_j\geq \alpha+2$.

    \textbf{Case $N_1\sim N_2\gtrsim N_3^*\gg N_4^*$: } Here we have (by symmetry) that $|\Omega_6|\gtrsim N_1N_3$ and the associated $M_6$ bound $|M_6|\lesssim m(N_1^*)N_1^*m(N_3^*)N_3^*$. We may now proceed with regularity, which again follows from an application of Lemma \ref{Lemma: Omega derivative}, but with the extension:
    \[
        \tilde{\Omega}_6(k_1,\dots, k_6) = -k_1k_3 - k_3(k_4+k_5+k_6)-(k_1-k_2)(k_4+\dots+k_6)-k_4^2+k_5^2-k_6^2,
    \]
    where $|\partial_{\xi_i}\tilde{\Omega}_6|\lesssim N_1$, $L_i \sim \min(N_i, N_3)$ and any $2\beta_j - \gamma_j \geq \alpha +2$:
    \[
    \frac{|\partial_{\xi_i}\tilde{\Omega}_p|^{\gamma_j}L_i^\alpha}{|\tilde{\Omega}_p|^{\beta_j-1}}\lesssim \frac{N_1^{\gamma_j}N_3^{\alpha}}{N_1^{\beta_j-1}N_3^{\beta_j-1}}\lesssim N_3^{\gamma_j-2\beta_j+\alpha+2}\lesssim 1.
    \]
    \textbf{Case $N_1^*\sim N_2^*\gg N_3^*\sim N_4^*$: }We have here that $|k_1+k_2|\gg \tfrac{(N_3^*)^2}{N_1^*}$, so we let $N_{12} \sim |k_1+k_2|$, and decompose $|k_i|\sim N_{12}$ for $i=1,2$. Moreover, because $N_{12}N_1^*\gg (N_3^*)^2$, we have
    \begin{align*}
    |M^1_6|&\lesssim m^2(N_1^*)N_1^*N_{12} + m^2(N_3^*)(N_{3}^*)^2\lesssim m^2(N_3^*)N_1^*N_{12}\\
    |\Omega_6|&\gtrsim N_1^*N_{12} + (N_3^*)^2\gtrsim N_1^*N_{12}.
    \end{align*}
    We take the trivial extension of $\Omega_6$, so that it suffices to check the regularity condition \eqref{Equation: Regularity Omega desired est}. Observe that $|\partial_{\xi_i}\tilde{\Omega}_p|\sim N_i$ and $L_i\sim \min(N_{12}, N_i)$, so that:
    \[
    \frac{|\partial_{\xi_i}\tilde{\Omega}_p|^{\gamma_j}L_i^\alpha}{|\tilde{\Omega}_p|^{\beta_j-1}}\sim \frac{N_i^{\gamma_j}\min(N_{12}, N_i)^\alpha}{(N_1^*)^{\beta_j-1}N_{12}^{\beta_j-1}}\lesssim N_i^{\gamma_j-\beta_j+1}N_{12}^{\alpha - \beta_j + 1}\lesssim N_{12}^{\gamma_j - 2\beta_j+\alpha+2}\lesssim 1,
    \]
    where we have first used $\beta \geq \gamma_j+1$, and then $2\beta_j-\gamma_j\geq \alpha + 2$.
    
    \textbf{Case $N_1^*\sim N_4^*\gg N_5^*$: }We have no contribution in the case that $\{N_1^*, N_2^*, N_3^*, N_4^*\} = \{N_1, N_2, N_3, N_4\},$ so we may restrict to either the cases:
    \begin{enumerate}[label = (\roman*)]
        \item $\{N_1^*, N_2^*, N_3^*, N_4^*\} = \{N_1,N_2,N_3,N_5\}$
        \item $\{N_1^*, N_2^*, N_3^*, N_4^*\} = \{N_1,N_2, N_4, N_6\}$.
    \end{enumerate}
    These are symmetric, so we handle them together by assuming the first. By the negation of the definition of resonance in this case, Definition \ref{Definition: Resonant Set}, we observe that one of the following is true:
    \begin{enumerate}
        \item[i1)] $k_1, k_3, k_5$ are of the same sign,
        \item[i2)] there is an $i\in \{1,3,5\}$ with $k_i$ having the same sign as $k_2$ and $|k_i-k_2|\ll N_1^*$,
        \item[i3)] there is an $i\in\{1,3,5\}$ with $k_i$ having the opposite sign as $k_2$ and $|k_i+k_2|\ll N_1^*$.
    \end{enumerate}
    In the first case we take 
    \[
    \tilde{\Omega}_6(k_1,\dots, k_6) = -2k_1k_3-2k_1k_5-2k_3k_5 - (k_1+k_3+k_5)(k_4+k_6)-k_4^2-k_6^2,
    \]
    while in all of the others we take the trivial extension
    \[
    \tilde{\Omega}_6(k_1,\dots, k_6) = (k_i-k_2)(k_i+k_2) + \sum_{\substack{j\ne i,2\\1\leq j\leq 6}}(-1)^{j+1}k_j^2,
    \]
    and observe that $|\tilde{\Omega}_6|\gtrsim (N_1^*)^2$ due to sign considerations. As the regularity estimate \eqref{Equation: Regularity Omega desired est} is trivial in this case, we may conclude by the observation that $|M^1_6/\Omega_6|\lesssim m^2(N_1^*)$.
\end{proof}
\subsection{Multilinear Estimates}\label{Section: I-method multilinear estimates}
In this section we obtain estimates on the multilinear operators required of Proposition \ref{Proposition: I method bounds}. In particular, will now estimate each multilinear operator of Equation \ref{Equation: I method FTC} utilizing Proposition \ref{Proposition: Regularity} and Lemmas \ref{Lemma: M6 estimate} and \ref{Lemma: spatial symbol size estimate}. In all of the work that is to come we ignore the presence of complex conjugates, as these change nothing outside of Lemma \ref{Lemma: M6 integral term}, where the convention will be that even indexed terms come with a complex conjugate.

The first term we estimate is the pointwise term involving the nonresonant $\tilde{\sigma}_6$. This estimate is the main culprit in the virtual limitation of $s > \frac13$, which arises when all of the frequencies are $\gtrsim N$. To overcome this issue we make use of the additional factors of $m$ to compensate for losses induced via Sobolev Embedding.
\begin{lemma}\label{Lemma: Pointwise tilde sigma 6}
    Let $s > 0$. Then the following estimate holds:
    \[
|\Lambda_6(\tilde{\sigma}_6)(t)|\lesssim N^{-\delta}\|Iu(t)\|_{H^1_x(\lambda\mathbb{T})}^6.
    \]
\end{lemma}
\begin{proof}
    This estimate hinges on an improvement over the trivial $\lesssim 1$ bound for the symbol $\tilde{\sigma}_6$. In particular, by a Littlewood-Paley decomposition, Proposition \ref{Proposition: Regularity}, and Lemma \ref{Lemma: spatial symbol size estimate} we see that it is sufficient to show that:
    \[
    m^2(N_3)\left|\int_{\lambda\mathbb{T}} \prod_{j=1}^6 P_{N_j}u_j\,dx\right|\lesssim N_1^{0-}\prod_{j=1}^6\|P_{N_j}Iu\|_{L^2_x(\lambda \mathbb{T})},
    \]
    where we assume that $N_1\sim N_2\gtrsim N_3\gtrsim\dots N_6$ and $N_2\gtrsim N$. However, this is nearly immediate from Sobolev Embedding:
    \[
        m(N_4)m(N_3)\left|\int_{\lambda\mathbb{T}} \prod_{j=1}^6 P_{N_j}u_j\,dx\right|
        \lesssim C(N_1,\dots, N_6)\prod_{j=1}^6\|P_{N_i}Iu\|_{H^1_x(\lambda\mathbb{T})},
    \]
    where
    \[
        C(N_1,\dots, N_6) = \frac{(N_3N_4N_5N_6)^{\frac12}}{N_3N_4\prod_{j\ne 3,4}m(N_i)N_i}
        \lesssim N^{-1+s}N_1^{-s}.
    \]
\end{proof}

We now handle the $10$-linear operator induced by the integration-by-parts procedure. The main contribution of the below is in the gains obtained by carrying the $m(N_1^*)m(N_3^*)$ bound over the trivial estimate of $1$. In particular, this enables us to handle case that $|k_1|,\dots, |k_{10}|\gtrsim N$, although care must be taken in the scenarios when we have substituted into the terms corresponding to $N_1^*$ and $N_3^*$ when forming the $10$-linear operator from the $6$-linear operator.

\begin{lemma}\label{Lemma: 10 linear operator}
    Let $s > 0$ and $T\lesssim \mu = \lambda^{2-\alpha}N^{-\alpha}$. Then:
    \begin{equation}\label{Equation: Lemma L10 integral}
\left|\int_0^T \Lambda_{10}(M_{10})\,ds\right|\lesssim N^{-3+}\|Iu\|_{Y_T^1}^{10}.
    \end{equation}
\end{lemma}
\begin{proof}
    Before proceeding, we note how we extend bounds for the $6$ linear operator to those of the $10$ linear operator. The main tool for this is \cite[Proposition 5.9]{schippa2024improved}, which we review for completeness (Actually this is the main use of Proposition \ref{Proposition: Regularity}). 
    
    For the sake of discussion, we consider $X_1$ and observe that this is a $10$-linear operator with symbol of the form
    \[
    \frac{M^1_6}{\Omega_6}(k_{11}+\dots + k_{15}, k_2, \dots, k_6).
    \]
    We perform a Littlewood-Paley decomposition and assume $|k_{11}+\dots +k_{15}| \sim N_1$ and $|k_2|\sim N_2,\dots, |k_6|\sim N_6$. Observe that by the regularity assumptions and Proposition \ref{Proposition: Regularity} that there is an extension $\overline{M_{10}}$ of $M^1_6/\Omega_6$ with
    \begin{multline*}
\overline{M_{10}}(k_{11}+\dots + k_{15}, k_2, \dots, k_6)\chi_{I_1}(k_{11}+\dots+k_{15})\chi_{I_2}(k_2)\dots\chi_{I_6}(k_6)\\
= \frac{1}{L_1\cdots L_6}\sum_{\xi_1,\dots, \xi_6\in\mathbb{Z}}m(\tfrac{\xi_1}{L_1},\dots,\tfrac{\xi_6}{L_6})e^{i\frac{(k_{11}+\dots+k_{15})\xi_1}{L_1}}\prod_{j=2}^6e^{i\frac{k_j\xi_j}{L_j}},
    \end{multline*}
    where $m$ satisfies the bounds provided for $M^1_6/\Omega$ so long as $k_{11}+\dots +k_{15}$ is localized to an appropriately sized interval. We then find via Fourier Inversion that
    \begin{multline}\label{Equation: 10 linear reduction firs step}
        \underset{\substack{k_1+\dots+k_6 = 0\\k_{11}+\dots+k_{15} = k_1\\0\leq t\leq T}}{\int\sum}\overline{M_{10}}(k_{11}+\dots+k_{15}, k_2,\dots, k_6)\widehat{u}(k_{11},t)\dots \widehat{u}(k_{15},t)\prod_{j=2}^6\widehat{u}(k_j,t)\,dt\\
        =\frac{1}{L_1\dots L_6}\sum_{\xi_1,\dots, \xi_6\in \mathbb{Z}}m(\tfrac{\xi_1}{L_1},\dots,\tfrac{\xi_6}{L_6})\int_{\lambda\mathbb{T}\times [0,T]}\prod_{j=1}^5u(x+\frac{\xi_{1}}{L_1},t)\prod_{j=2}^6u(x+\tfrac{\xi_j}{L_j},t)\,dx\,dt,
    \end{multline}
    where this holds under localization of the individual $u$'s in the same manner. It follows that we may take the claimed bounds above, even for the larger multilinear operator. A word of warning: even though we're taking bounds with respect to an extension of $M^1_6/\Omega_6$ viewed as a symbol on $\mathbb{R}^6$, we will still denote the bound of Proposition \ref{Proposition: Regularity} for this object as $\overline{\mathcal{M}_{10}}$.

    By further dyadically decomposing $|k_{1j}|\sim K_{1j}$, it follows that
    \begin{multline*}
        \eqref{Equation: Lemma L10 integral}\lesssim \sum_{K_{11}\gtrsim \cdots K_{15}}\sum_{N_1\gtrsim \cdots \gtrsim N_6}\frac{\overline{\mathcal{M}_{10}}}{P_{N_1}\prod_{j=1}^5m(K_{1j})K_{1j}\prod_{j=2}^6m(N_j)N_j}\\
        \times\left|\int_{\lambda\mathbb{T}}\int_0^T P_{N_1}\left(\prod_{j=1}^5 w_{K_{1j}}\right)\prod_{j=2}^6w_{N_j}\,dtdx\right|,
    \end{multline*}
    with the obvious modifications when we wish to utilize a finer decomposition and almost-orthogonality.
    
    We now observe that if  $N_1^*\sim N_2^*\gg N_3^*\sim N_4^*$ and $N_1 \in \{N_3^*, \dots, N_6^*\}$, then may assume that $N_1 = N_3^*$ and localize both $|k_1^*|\sim N_{12}$ and $|k_2^*|\sim N_{12}$ with $|k_1^*+k_2^*|\sim N_{12}$ (where $k_j^*$ for $j = 1,2$ are taken with respect to the collection $\{k_2,\dots, k_6\}$). Under this convention, we have the size and regularity bound $m^{\frac32}(N_4^*)m^{\frac12}(N_{12})$. We may now perform a relabeling: we now deal with dyadic frequencies $K_1,\dots, K_{10}$, where $K_1 = N_1^*$, $K_2 = N_2^*$ and the remaining are arbitrarily decided, and denote $K_{12} = N_{12}$. Taking instead the weaker size bound 
    \[
    \overline{\mathcal{M}_{10}}\lesssim m^\frac32(N_4^*)m^\frac12(N_{12})\lesssim m^\frac32(K_8^*)m^\frac12(K_{12}),
    \]
    we see that this translates into the bound:
    \begin{multline}\label{Equation: Worse Case}
\left|\int_{\Gamma_6}\int_0^T\frac{\overline{M_{10}}(k_{11}+\dots+k_{15},k_2,\dots,k_6)\prod_{j=1}^5\widehat{w}(k_{1j},t)\prod_{j=2}^6\widehat{w}(k_j,t)}{\prod_{j=1}^5m(k_{1j})\langle k_{1j}\rangle\prod_{j=2}^6m(k_j)\langle k_j\rangle}\right|\\
\lesssim 
\sum_{K_1^*\gtrsim \cdots K_{10}^*}\sum_{K_{12}\lesssim K_1}\sum_{\substack{I\sim J\\ |I_1|,|I_2|=K_{12}}}C_1(K_{12},K_1,\dots, K_{10})\\
\qquad\qquad\times\left|\int_{\lambda\mathbb{T}\times [0,T]}P_{I}w_{K_1}P_{J}w_{K_2}\prod_{j=3}^{10}w_{K_j}\,dxdt\right|,
    \end{multline}
    where
    \[
C_1(K_{12},K_1^*,\dots, K_{10}^*) = \frac{m^{\frac32}(K_{8}^*)m^{\frac12}(K_{12})}{\prod_{j=1}^{10}m(K_j)K_j}.
    \]
    
    If we're not in this case, then we know that either $N_1^*\sim N_2^*\gg N_3^*\sim N_4^*$ and $N_1 \in \{N_1^*, N_2^*\}$, in which case we have the bound
    \[
    \left|\frac{M^1_6}{\Omega_6}(k_{11}+\dots + k_{15}, k_2, \dots, k_6)\right|\lesssim \overline{\mathcal{M}_{10}}\lesssim m(N_3^*)m(N_4^*),
    \]
    or we have
    \[  
    \left|\frac{M^1_6}{\Omega_6}(k_{11}+\dots + k_{15}, k_2, \dots, k_6)\right|\lesssim \overline{\mathcal{M}_{10}}\lesssim m(N_1^*)m(N_3^*),
    \]
    with no constraint on the $j$ such that $N_1 = N_j^*$. However, in each of these cases we observe that we may relabel the collection of $K_{1j}$ and $N_j$ into $K_1, \dots, K_{10}$ and take the bound
    \[
    \overline{\mathcal{M}_{10}}\lesssim m(K_7^*)m(K_8^*).
    \]
    In this case we then find
    \begin{multline}\label{Equation: Better Case}
        \left|\int_{\Gamma_6}\int_0^T\frac{\overline{M_{10}}(k_{11}+\dots+k_{15},k_2,\dots,k_6)\prod_{j=1}^5\widehat{w}(k_{1j},t)\prod_{j=2}^6\widehat{w}(k_j,t)}{\prod_{j=1}^5m(k_{1j})\langle k_{1j}\rangle\prod_{j=2}^6m(k_j)\langle k_j\rangle}\right|\\
        \lesssim 
        \sum_{K_1^*\gtrsim \cdots K_{10}^*}C_2(K_1,K_1,\dots, K_{10})\left|\int_{\lambda\mathbb{T}\times [0,T]}\prod_{j=1}^{10}w_{K_j}\,dxdt\right|,
    \end{multline}
    where
    \[
        C_2(K_1,K_1,\dots, K_{10}) := \frac{m(K_{7}^*)m(K_{8}^*)}{\prod_{j=1}^{10}m(K_j)K_j}.
    \]

    We now observe that in either \eqref{Equation: Worse Case} or \eqref{Equation: Better Case}, we may assume that $K_4^*\gtrsim N$. Indeed, if $K_4^*\ll N$, then 
    \begin{equation}\label{Equation: C1C2 low bound}
        C_1,C_2\lesssim \frac{1}{K_3^*\cdots K_{10}^*}
            \begin{cases}
            N^{-2+2s}(K_1^*K_2^*)^{-s} & \mbox{ if }K_3^*\ll N\\
            N^{-3+3s}(K_1^*K_2^*K_3^*)^{-s}K_3^* & \mbox{ if }K_3^*\gtrsim N\gg K_4^*,
            \end{cases}
    \end{equation}
    where we see that we only have to gain $N$ in the first case, and don't have to gain anything in the second. We may handle the first case after applying Corollary \ref{Corollary: Trilinear Strichartz Specific} twice
    \begin{multline*}
        \left|\int_{\lambda\mathbb{T}\times [0,T]}\prod_{j=1}^{10}w_{K_j}\,dxdt\right|\lesssim \|w_{K_1^*}w_{K_3^*}w_{K_5^*}\|_{L^2_{x,t}}\|w_{K_2^*}w_{K_4^*}w_{K_6^*}\|_{L^2_{x,t}}\prod_{j=7}^{10}\|w_{K_j^*}\|_{L^\infty_{x,t}}\\
        \lesssim (K_1^*)^{0+}\left(\frac{K_3^*}{K_1^*} + \frac{1}{N}\right)(K_7^*\cdots K_{10}^*)^{\frac12}\prod_{j=1}^{10}\|w_{K_j}\|_{Y^0},
    \end{multline*}
    where summation yields the result after noting that \eqref{Equation: C1C2 low bound} allows for losses in $K_3^*$. The second case can be handled without the use of the trilinear estimate, by simply utilizing the $L^6_{x,t}$ estimate of Proposition \ref{Proposition: Scaled Strichartz}
    \begin{multline*}
        \left|\int_{\lambda\mathbb{T}\times [0,T]}\prod_{j=1}^{10}w_{K_j}\,dxdt\right|
        \lesssim \prod_{j=1}^6 \|w_{K_j^*}\|_{L^6_{x,t}}\prod_{j=7}^{10}\|w_{K_j^*}\|_{L^\infty_{x,t}}
        \lesssim (K_1^*)^{0+}(K_7^*\cdots K_{10}^*)^{\frac12}\prod_{j=1}^{10}\|w_{K_j}\|_{Y^0},
    \end{multline*}
    and a similar summation.

    We now may assume that $K_4^*\gtrsim N$, and proceed with estimating \eqref{Equation: Worse Case}. Here, we recall that $N_1 = K_1$ and $N_2 = K_{2}$ and that we've localized $w_{K_1}$ and $w_{K_2}$ to intervals of size $K_{12} = N_{12}$. We may apply the $L^6$ estimate to the terms corresponding to $K_2, K_3^*\cdots K_7^*$, and Sobolev embedding to the terms corresponding to $K_1, K_8^*,\dots K_{10}^*$. In this fashion we find
    \begin{multline}\label{Equation: Worst case expanded}
    \eqref{Equation: Worse Case}\lesssim \sum_{K_1^*\gtrsim \cdots \gtrsim K_{10}^*}\sum_{K_{12}\lesssim K_1}\sum_{\substack{I\sim J\\ |I|,|j|=K_{12}}}C_1\cdot(K_8^*K_9^*K_{10}^*)^{\frac12}(K_1^*)^{0+}\\
    \times \|P_{I}w_{K_1}\|_{L^\infty_{x,t}}\|P_{J}w_{K_2}\|_{Y^0}\prod_{j=3}^{10}\|w_{K_j}\|_{Y^0}\\
    \lesssim \sum_{K_1^*\gtrsim \cdots \gtrsim K_{10}^*}\sum_{K_{12}\lesssim K_1}\sum_{\substack{I\sim J\\ |I|,|j|=K_{12}}}C_1\cdot(K_8^*K_9^*K_{10}^*)^{\frac12}K_{12}^{\frac12}(K_1^*)^{0+}\\
    \times \|P_{I}w_{K_1}\|_{Y^0}\|P_{J}w_{K_2}\|_{Y^0}\prod_{j=3}^{10}\|w_{K_j}\|_{Y^0},
    \end{multline}
    where
    \[
C_1\cdot(K_8^*K_9^*K_{10}^*)^{\frac12}K_{12}^{\frac12}(K_1^*)^{0+}\lesssim \frac{(m(K_8^*)K_8^*m(K_{12})K_{12})^{\frac12}}{N^{4-4s}(K_1^*K_2^*K_3^*K_4^*)^{s-}}\lesssim N^{-3+}(K_1^*)^{0-}.
    \]
    Summation of the Dyadic blocks then yields the result.

    On the other hand, if we are to estimate \eqref{Equation: Better Case}, then we have dyadic blocks $K_1^*\gtrsim \cdots \gtrsim K_{10}^*$, and we apply the $L^6_{x,t}$ estimate to $w_{K_1^*},\dots, w_{K_6^*}$ while Sobolev embedding to the remaining. In this way, we observe:
    \[
        \eqref{Equation: Better Case}\lesssim \sum_{K_1^*\gtrsim \cdots \gtrsim K_{10}^*}C_2\cdot(K_7\cdots K_{10})^{\frac12}(K_1)^{0+}\prod_{j=1}^{10}\|Iu_{K_j}\|_{Y^1},
    \]
    where
    \[
C_2\cdot(K_7^*\cdots K_{10}^*)^{\frac12}(K_1^*)^{0+}\lesssim N^{-4+}(K_1^*)^{0-}\frac{(K_7^*\cdots K_{10}^*)^{\frac12}}{K_7^*K_8^*}\lesssim N^{-4+}(K_1^*)^{0-},
    \]
    and the result follows again by summation of the dyadic blocks.
\end{proof}
The final lemma we prove provides a bound on $\Lambda_6(\overline{M}_6)$, which is the resonant $6$-linear operator. We recommend the reader refer to Definition \ref{Definition: Resonant Set} for a review of the resonance set, and the associated regularity and size bounds of Lemma \ref{Lemma: M6 estimate}. This again differs from \cite[Proposition 5.7]{schippa2024improved} in the reliance on the trilinear estimate of Proposition \ref{Lemma: trilinear Estimates}, with much of the work included for clarity.
\begin{lemma}\label{Lemma: M6 integral term}
If $s > 0$ and $T\lesssim \mu = \lambda^{2-\alpha}N^{-\alpha}$, then
\begin{equation}\label{Equation: L6 lemma integral bound}
\left|\int_0^T\Lambda_6(\overline{M}_6)\,dt\right|\lesssim N^{-3+}\|Iu\|_{Y^1_T}^6.
\end{equation}
\end{lemma}
\begin{proof}
    We dyadically decompose into $|k_i|\sim N_i$, and let $N_j^*$ denote the decreasing rearrangement of the terms. By symmetry we may assume $N_1\geq N_2$ and $N_1\geq N_3\geq N_5$, $N_2\geq N_4\geq N_6$, and our assumptions on resonance afford us $\{N_1^*, N_2^*\} = \{N_1, N_2\}$. We again denote $\widehat{w}(\xi,t) = m(\xi)\langle\xi\rangle \widehat{u}(\xi,t)$, and reiterate that all even indexed terms implicitly come with complex conjugation. With this notation, the fundamental building block will then be the estimate
    \begin{multline}\label{Equation: Resonant Base Estimate}
        \left|\int_{\Gamma_6}\int_0^T\frac{\overline{M}_6(k_1,\dots, k_6)\prod_{j=1}^6\widehat{w}(k_j,t)}{\prod_{j=1}^6m(k_j)\langle k_j\rangle}\,dtd\Gamma_6\right|\\
        \lesssim \sum_{\substack{I_1,\dots, I_6\\I_i\subset A_{N_i}\\(\mathcal{R})\,\mathrm{ holds}}}\frac{\overline{\mathcal{M}}_6(I_1,\dots, I_6)}{\prod_{j=1}^jm(N_j)N_j}\left|\int_{\Gamma_6}\int_0^T\prod_{j=1}^6\widehat{w}(k_j,t)\chi_{I_j}(k_j)\,dtd\Gamma_6\right|,
    \end{multline}
    where we obtain bounds for $\overline{\mathcal{M}}_6$ via Lemma \ref{Lemma: M6 estimate}. This lemma will also provide us the regions $I_j$ on which we will sum. In most of the scenarios we have that the blocks $I_j$ are the dyadic $N_j$, but in $A_1$ and portions of $A_3$ they will be slightly smaller in accordance with Lemma \ref{Lemma: M6 estimate}.

    As is standard, we split the integral into regions
    \begin{align*}
        A_1 &= \{k\in\Gamma_6\,:\,N_2^*\gtrsim N\gg N_3^*\sim N_4^*\},\\
        A_2 &= \{k\in \Gamma_6\,:\, N_3^*\gtrsim N\gg N_4^*\},\\
        A_3 &= \{k\in\Gamma_6\,:\,N_4^*\gtrsim N\gg N_5^*\},\\
        A_4 &= \{k\in\Gamma_6\,:\, N_5^*\gtrsim N\},
    \end{align*}
    where there is no contribution from $A_2$ in this integral.\\
    
    \textbf{Case $A_1$: } In this region we take the $\overline{\mathcal{M}}_6$ bound of Equation \eqref{Equation: M bound 2}, $\overline{\mathcal{M}}_6\lesssim (N_3^*)^2$. We observe that we may apply the trilinear Strichartz of Corollary \ref{Corollary: Trilinear Strichartz Specific}. Hence, we find
    \begin{multline*}
        \sum_{\substack{I_1,\dots, I_6\\I_i\subset A_{N_i}\\(\mathcal{R})\,\mathrm{ holds}}}\frac{\overline{\mathcal{M}}_6(I_1,\dots, I_6)}{\prod_{j=1}^jm(N_j)N_j}\left|\int_{\lambda\mathbb{T}}\int_0^T\prod_{j=1}^6w_{N_j^*}\,dtdx\right|\\
        \lesssim \sum_{N_1^*\sim N_2^*\gg N_3^*\gtrsim\cdots\gtrsim N_6^*}\frac{N_3^*N_4^*}{N^{2-2s}(N_1^*)^{2s-}\prod_{j=3}^6N_j^*}\|w_{N_1^*}w_{N_3^*}w_{N_5^*}\|_{L^2_{x,t}}\|w_{N_2^*}w_{N_4^*}w_{N_6^*}\|_{L^2_{x,t}}\\
        \lesssim \sum_{N_1^*\sim N_2^*\gg N_3^*\gtrsim\cdots\gtrsim N_6^*}\frac{N_3^*N_4^*}{N^{2-2s}(N_1^*)^{2s-}N_3^*N_4^*(N_5^*N_6^*)^\frac{1}{2}}\left(\frac{1}{N_1^*} + \frac{1}{N}\right)\prod_{j=1}^6\|w_{N_j^*}\|_{Y^0}\\
        \lesssim \sum_{N_1^*\sim N_2^*\gg N_3^*\gtrsim\cdots\gtrsim N_6^*}N^{-2+}\left(\frac{1}{N_1^*} + \frac{1}{N}\right)\prod_{j=1}^6\|w_{N_j^*}\|_{Y^0}\lesssim N^{-3+}\|Iu\|_{Y^1_T}^6,
    \end{multline*}
    as desired. This also provides the blueprint for the remaining estimates, in that we will seek to always apply two trilinear estimates, unless we have enough high frequency terms to compensate.\\

    \textbf{Case $A_3$: } We further split $A_3$ into two regions, given as
    \begin{align*}
        A_{31} &= \{k\in A_3\,:\,N_1^*\gg N_3^*\}\\
        A_{32} &= \{k\in A_3\,:\, N_1^*\sim N_3^*\}.
    \end{align*}
    \textbf{Case $A_{31}$:} In this case we observe that
    \[
    |k_1^*\pm k_3^*|,\,|k_2^*\pm k_4^*|\gtrsim N_1^*,
    \]
    and that we may perform the almost orthogonal decomposition of $k_j\in I_j$ for $j = 1,2$ with $|I_j|\sim N_3^*$ so that we have the estimate $\overline{\mathcal{M}}_6\lesssim m(N_1^*)N_1^*m(N_3^*)N_3^*$ by equation \eqref{equation: M bound 1}. This will then follow from the observation that we may perform two applications of the trilinear Strichartz estimate of Corollary~\ref{Corollary: Trilinear Strichartz Specific}, which yields
    \begin{multline*}
    \|w_{N_5^*}w_{N_3^*}P_{I_1}w_{N_1^*}\|_{L^2_{x,t}}\|w_{N_6^*}w_{N_4^*}P_{I_2}w_{N_2^*}\|_{L^2_{x,t}}\\
    \lesssim (N_1^*)^{0+}\left(\frac{N_5^*}{N_1^*}+\frac{1}{N}\right)\prod_{j=1}^2 \|P_{I_j}w_{N_j^*}\|_{Y^0}\prod_{j=3}^6\|w_{N_j^*}\|_{Y^0},
    \end{multline*}
    and hence
    \begin{multline*}
        \eqref{Equation: Resonant Base Estimate}\lesssim \sum_{\substack{N_1^*\sim N_2^*\gg N_3^*>\dots> N_6^*\\I_1\sim I_2\\|I_1|,\,|I_2|\sim N_3^*}}\frac{N^{-2+2s}(N_{2}^*N_4^*)^{-s}}{N_5^*N_6^*}\left|\int_{\lambda\mathbb{T}}\int_0^T\prod_{j=1}^2P_{I_1}w_{N_j^*}\prod_{j=3}^6w_{N_j^*}\,dtd\Gamma_6\right|\\
        \lesssim \sum_{\substack{N_1^*\sim N_2^*\gg N_3^*>\dots> N_6^*\\I_1\sim I_2\\|I_1|,\,|I_2|\sim N_3^*}}\frac{N^{-2+2s}(N_{2}^*N_4^*)^{-s+}}{N_5^*N_6^*}\|w_{N_5^*}w_{N_3^*}P_{I_1}w_{N_1^*}\|_{L^2_{x,t}}\|w_{N_6^*}w_{N_4^*}P_{I_2}w_{N_2^*}\|_{L^2_{x,t}}\\
        \lesssim\sum_{N_1^*\sim N_2^*\gg N_3^*>\dots> N_6^*}\frac{N^{-2+2s}(N_{2}^*N_4^*)^{-s+}}{(N_5^*N_6^*)^{\frac12}}\left(\frac{1}{N_1^*} + \frac{1}{N}\right)\prod_{j=1}^6 \|Iu_{N_j^*}\|_{Y^1}\\
        \lesssim N^{-3+}\|Iu\|_{Y^1_T}^6.
    \end{multline*}
    
    \textbf{Case $A_{32}$:} Turning our attention to $A_{32}$ we have that $N_3^*\sim N_4^*$ and hence $N_1^*\sim N_4^*$, which makes the application of two trilinear estimates considerably more complicated.

    To overcome this, we decompose $A_{32}$ into
    \begin{align*}
        A_{321} &= \{k\in A_{32}\,:\,\{N_1^*,N_2^*,N_3^*,N_4^*\} = \{N_1,N_2,N_3,N_4\}\}\\
        A_{322} &= \{k\in A_{32}\,:\,\{N_1^*,N_2^*,N_3^*,N_4^*\} = \{N_1,N_2,N_3,N_5\}\}\\
        A_{321} &= \{k\in A_{32}\,:\,\{N_1^*,N_2^*,N_3^*,N_4^*\} = \{N_1,N_2,N_4,N_6\}\}.
    \end{align*}

    \textbf{Subcase $A_{321}$: } To utilize the trilinear Strichartz estimate here, we need to perform our (final) decomposition into the sets:
    \begin{align*}
        A_{3211} &= \{k\in A_{321}\,:\,k_1 >0,\, k_2 > 0,\, k_3 < 0,\, k_4 < 0\}\\
        A_{3212} &= \{k\in A_{321}\,:\,k_1 >0,\, k_2 < 0,\, k_3 < 0,\, k_4 > 0\}\\
        A_{3213} &= \{k\in A_{321}\,:\,k_1 >0,\, k_2 < 0,\, k_3 < 0,\, k_4 < 0\}\\
        A_{3214} &= \{k\in A_{321}\,:\,k_1 >0,\, k_2 < 0,\, k_3 > 0,\, k_4 < 0\},
    \end{align*}
    where we have used symmetry to assume that $k_1 > 0$. Note that we've possibly lost the original order of $N_1\geq N_3$ and $N_2\geq N_4$ in reducing the total number of possibilities to the above four, but the proof will be insensitive to this.

    We first observe that in either $A_{3211}$ and $A_{3212}$ we can perform two trilinear Strichartz estimates on the terms $u_1u_3u_5$ and $u_2u_4u_6$. Utilizing the bound from Equation \eqref{equation: M bound 1}, we have the size and regularity estimate
    \[
    \overline{\mathcal{M}}(I_1,\dots, I_6)\lesssim m(N_1^*)N_1^*m(N_3^*)N_3^*.
    \]
    Observing that
    \[
    |k_1-k_3|\sim |k_2-k_4|\sim |k_5-k_1|\sim |k_6-k_2|\sim N_1^*,
    \]
    we see that we may apply Corollary \ref{Corollary: Trilinear Strichartz Specific} twice, obtaining
    \begin{multline*}
        \eqref{Equation: Resonant Base Estimate}\lesssim \sum_{N_1\sim N_4\gtrsim N\gg N_5^*\geq N_6^*}\frac{N^{2s-2}(N_1)^{-2s}}{N_5N_6}\left|\int_{\lambda\mathbb{T}\times [0,T]}\prod_{j=1}^6w_{N_j^*}\,dxdt\right|\\
        \lesssim \sum_{N_1^*\sim N_4^*\gtrsim N\gg N_5^*}\frac{N^{2s-2}(N_1^*)^{-2s}}{N_5N_6}\|w_{N_1}w_{N_3}w_{N_5}\|_{L^2_{x,t}}\|w_{N_2}w_{N_4}w_{N_6}\|_{L^2_{x,t}}\\
        \lesssim \sum_{N_1\sim N_4\gtrsim N\gg N_5^*\geq N_6^*}\frac{N^{2s-3}(N_1^*)^{-2s+}}{(N_5N_6)^{\frac12}}\prod_{j=1}^6\|Iu_{N_j^*}\|_{Y^1}\lesssim N^{-3+}\|Iu\|_{Y^1_T}^6.
    \end{multline*}
    Suppose now that we are in case $A_{3213}$, in which case we have
    \[
    |k_2+k_3| = |k_1+k_4 + O(N_5^*)|\gtrsim N_1^*\gtrsim N,
    \]
    so that $|k_1+k_4|\gtrsim N_1^*$. In particular, we have access to two trilinear Strichartz applications and can apply the above argument again after a permutation of the functions.

    We are thus left with $A_{3214}$, where we distinguish two (further) cases:
    \begin{align*}
        A_{3214a} &:= \{k\in A_{3214}\,:\,|k_1+k_2|\lesssim |k_5^*|\}\\
        A_{3214b} &:= \{k\in A_{3214}\,:\,|k_1+k_2|\gg |k_5^*|\gtrsim |k_5+k_6|\}.
    \end{align*}
    In case $A_{3214a}$ we have that the linear relationship between $k_j$ forces $|k_3+k_4|\lesssim |k_5^*|$, and hence we may perform the further almost orthogonal decomposition of $k_j\in I_j$ for $j=1,2,3,4$ and $|I_j|\sim N_5^*$. Equation \eqref{Equation: M bound 3} yields the bound $\overline{\mathcal{M}}_6\lesssim m(N_1^*)N_1^*m(N_5^*)N_5^*$ in this scenario, and hence
    \begin{multline*}
        \eqref{Equation: Resonant Base Estimate}\lesssim \sum_{N_1\sim N_4\gtrsim N\gg N_5^*}\sum_{\substack{I_1\sim I_2\\ I_3\sim I_4\\|I_j|\sim N_5^*}}\frac{N^{3s-3}(N_1^*)^{-3s}}{N_6^*}\left|\int_{\lambda\mathbb{T}\times [0,T]}\prod_{j=1}^4P_{I_j}w_{N_j}\prod_{j=5}^6w_{N_j}\,dxdt\right|\\
        \lesssim \sum_{N_1\sim N_4\gtrsim N\gg N_5^*}\sum_{\substack{I_1\sim I_2\\ I_3\sim I_4\\|I_j|\sim N_5^*}}\frac{N^{3s-3}(N_1^*)^{-3s}}{N_6^*}\prod_{j=1}^4\|P_{I_j}w_{N_j}\|_{L^6}\prod_{j=5}^6\|w_{N_j}\|_{L^6}\\
        \lesssim N^{-3+}\|Iu\|_{Y^1_T}^6.
    \end{multline*}

    This leaves us with the final case of $A_{3214b}$, where $|k_1+k_2|\gg N_5^*\gtrsim |k_5+k_6|$. This forces $|k_3+k_4|\sim |k_1+k_2|$, so we perform a dyadic decomposition of $|k_1+k_2|\sim N_{12}$, and the further almost orthogonal decomposition of $k_j\in I_j$ for $|I_j|\sim N_{12}$ and $j = 1,2,3,4$. We note the bound $\overline{\mathcal{M}}_6\lesssim m(N_1^*)N_1^*m(N_{12})N_{12}$ of Equation \eqref{Equation: M bound 4}, which reduces us to bounding
    \begin{multline*}
        \eqref{Equation: Resonant Base Estimate}\lesssim \sum_{N_1^*\sim N_4^*\gtrsim N\gg N_5^*}\sum_{N_{12}\lesssim N_1^*}\sum_{\substack{I_1\sim I_2\\I_3\sim I_4\\|I_j|\sim N_{12}}}\frac{N^{3s-3}(N_1^*)^{-3s}m(N_{12})N_{12}}{N_5N_6}\\
        \times\left|\int_{\lambda\mathbb{T}\times [0,T]}\prod_{j=1}^4P_{I_j}w_{N_j}\prod_{j=5}^6w_{N_j^*}\,dxdt\right|.
    \end{multline*}
    In order to handle the $6$-linear estimate, we consider first the case that
    \[
    N_{12}^2\lesssim N_5^*N_1,
    \]
     and apply the base trilinear estimate of Proposition \ref{Lemma: trilinear Estimates} twice. This affords us
    \[
    \left|\int_{\lambda\mathbb{T}\times [0,T]}\prod_{j=1}^4P_{I_j}w_{N_j}\prod_{j=5}^6w_{N_j^*}\,dxdt\right|\lesssim (N_1^*)^{0+}\left(\frac{N_{12}}{N_1}+\frac{1}{N}\right)\prod_{j=1}^4\|P_{I_j}w_{N_j}\|_{Y^0}\prod_{j=5}^6\|w_{N_j}\|_{Y^0},
    \]
    where we may further sum to obtain
    \[
    \sum_{N_{12}^2\lesssim N_5^*N_1}(N_1^*)^{0+}m(N_{12})N_{12}\left(\frac{N_{12}}{N_1}+\frac{1}{N}\right)\lesssim (N_1^*)^{0+}N_5^* + N^{-s}N_1^{s+},
    \]
    by Remark \ref{Remark: N12 M term} and the summation restriction. If, on the other hand, we have
    \[
    N_5^*N_1\ll N_{12}^2,
    \]
    then we may apply the enhanced version of Proposition \ref{Lemma: trilinear Estimates} where
    \[
    K\lesssim \frac{N_5^*N_1}{N_{12}},
    \]
    and hence
    \[
    \left|\int_{\lambda\mathbb{T}\times [0,T]}\prod_{j=1}^4P_{I_j}w_{N_j}\prod_{j=5}^6w_{N_j^*}\,dxdt\right|\lesssim (N_1^*)^{0+}\left(\frac{N_5^*}{N_{12}}+\frac{1}{N}\right)\prod_{j=1}^4\|P_{I_j}w_{N_j}\|_{Y^0}\prod_{j=5}^6\|w_{N_j}\|_{Y^0}.
    \]
    Summation again yields the bound
    \[
    \sum_{N_5^*N_1\ll N_{12}^2\lesssim N_1^2}(N_1^*)^{0+}m(N_{12})N_{12}\left(\frac{N_{5}^*}{N_{12}}+\frac{1}{N}\right)\lesssim (N_1^*)^{0+}N_5^* + N^{-s}N_1^{s+}.
    \]
    It follows that in either situation we have
    \begin{multline*}
        \sum_{N_1^*\sim N_4^*\gtrsim N\gg N_5^*}\sum_{N_{12}\lesssim N_1^*}\sum_{\substack{I_1\sim I_2\\I_3\sim I_4\\|I_j|\sim N_{12}}}\frac{m(N_{12})N_{12}}{N^{3-3s}N_1^{3s-}N_5N_6}
        \left|\int_{\lambda\mathbb{T}\times [0,T]}\prod_{j=1}^4P_{I_j}w_{N_j}\prod_{j=5}^6w_{N_j^*}\,dxdt\right|\\
        \lesssim \sum_{N_1^*\sim N_4^*\gtrsim N\gg N_5^*}N^{-3+}N_1^{0-}\prod_{j=1}^6 \|w_{N_j}\|_{Y^0}\lesssim N^{-3+}\|Iu\|_{Y^1}^6,
    \end{multline*}
    which is our desired bound.

    \textbf{Subcases $A_{322}$ \& $A_{323}$: }We assume by symmetry of the conditions of resonance that we are in case $A_{322}$. We note that we must have that $k_1,k_3,k_5$ are not of the same sign. Moreover, if $k_i$ and $k_2$ are of the same sign then $|k_i-k_2|\sim N_1^*$, and $|k_i+k_2|\sim N_{1}^*$ if they have opposite signs. 
    
    We can assume without loss of generality that $k_1 > 0$. We then have the following cases:
    \begin{enumerate}[label = (\roman*)]
        \item $k_3 > 0,\,k_5 < 0$ or $k_3 < 0,\,k_5 > 0$
        \item $k_3 < 0$ \& $k_5 < 0$.
    \end{enumerate}
    We first note that regardless of the sign of $k_2$, we always have $|k_j+k_2|\gtrsim N_1^*$ for $j\in\{1,3,5\}$. Indeed, if they both have the same sign then it's immediate, while if they differ in signs the resonance condition will force $|k_j+k_2|\gtrsim N_1^*$. Now, if we fall in case $(i)$ then we have
    \[
    |k_1+k_2|,\,|k_3-k_5|\gtrsim N_1^*,
    \]
    so that we may apply two trilinear estimates to $w_{N_1}w_{N_2}w_{N_4}$ and $w_{N_3}w_{N_5}w_{N_6}$. If we fall in case $(ii)$ then
    \[
    |k_1-k_3|,\,|k_2+k_5|\gtrsim N_1^*,
    \]
    and we may apply two trilinear estimates to $w_{N_1}w_{N_3}w_{N_4}$ and $w_{N_2}w_{N_5}w_{N_6}$.

    Thus, in either case we may apply the bound $\overline{\mathcal{M}}_6\lesssim m(N_1^*)N_1^*m(N_3^*)N_3^*$ together with the two applications of Corollary \ref{Corollary: Trilinear Strichartz Specific} to obtain the bound
    \[
\eqref{Equation: Resonant Base Estimate}\lesssim N^{-3+}\|Iu\|_{Y^1_T}^6,
    \]
    as in the proof of $A_{3211}$.\\
    \textbf{Case $A_4$:} This is the final case, in which all frequencies are large. Here we just assume that $N_1\sim N_2\geq N_3\geq \dots\geq N_6$ and take the $\overline{\mathcal{M}}_6$ bound afforded by \eqref{equation: M bound 1}:
    \[
    \overline{\mathcal{M}}\lesssim m(N_1)N_1m(N_3)N_3.
    \]
    In this manner, we find via a straightforward $L^6$ estimate:
    \[
        \eqref{Equation: Resonant Base Estimate}\lesssim \sum_{N_1\sim N_2\geq N_3\dots N_6}N^{-3+3s}(N_2N_4N_5)^{-s}\prod_{j=1}^6\|w_{N_j}\|_{L^6_{x,t}}\lesssim N^{-3+}\|Iu\|_{Y^1_T}^6,
    \]
    which was our desired estimate.
\end{proof}

\section*{Acknowledgements}
The author would like to graciously thank the anonymous referee for their close reading and attentive comments, which helped to clarify a central argument and strengthen the presentation of the manuscript.

\end{document}